\documentclass
[
    a4paper,
    DIV=11,
    abstract=true,
]
{scrartcl}

\usepackage
{
    graphicx,
    amssymb,
    amsmath,
    amsthm,
    dsfont, 
    xcolor,
    mathtools,
    authblk,
    enumitem,
    tikz,
    nicefrac,
    mdframed, 
    tabularx,
    multicol,
}

\usepackage[scr=boondoxupr]{mathalfa}

\usepackage[utf8]{inputenc}

\usepackage[pdffitwindow=false,
            plainpages=false,
            pdfpagelabels=true,
            pdfpagemode=UseOutlines,
            pdfpagelayout=SinglePage,
            bookmarks=false,
            colorlinks=true,
            hyperfootnotes=false,
            linkcolor=blue,
            urlcolor=blue!30!black,
            citecolor=green!50!black]{hyperref}

\usepackage[bf,normal]{caption}

\usepackage[square,sort,comma,numbers]{natbib}
\usetikzlibrary{arrows, arrows.meta, shapes, fit, automata, positioning}

\graphicspath{{pics/}}

\DeclareMathAlphabet{\mathpzc}{OT1}{pzc}{m}{it}

\newcommand{\subfiguretitle}[1]{{\scriptsize{#1}} \\}
\newcommand{\R}{\mathbb{R}}                                      
\newcommand{\pd}[2]{\frac{\partial#1}{\partial#2}}               
\newcommand{\ts}{\hspace*{0.1em}}                                
\newcommand{\mc}[2][]{\mathpzc{#2}{\smash[t]{\mathstrut}}_{#1}}  

\newcommand\xqed[1]{\leavevmode\unskip\penalty9999 \hbox{}\nobreak\hfill \quad\hbox{#1}}
\newcommand{\exampleSymbol}{\xqed{$\triangle$}}

\DeclareMathOperator{\diag}{diag}

\newtheorem{theorem}{Theorem}[section]

\newtheorem{lemma}[theorem]{Lemma}
\newtheorem{proposition}[theorem]{Proposition}
\newtheorem{definition}[theorem]{Definition}
\theoremstyle{definition}
\newtheorem{example}[theorem]{Example}
\newtheorem{remark}[theorem]{Remark}
\newtheorem{textalgorithm}[theorem]{Algorithm}

\makeatletter
\renewcommand*\env@matrix[1][*\c@MaxMatrixCols c]{%
  \hskip -\arraycolsep
  \let\@ifnextchar\new@ifnextchar
  \array{#1}}
\makeatother

\mathtoolsset{centercolon} 

\definecolor{boxback}{gray}{0.95}

\allowdisplaybreaks

\DeclareCaptionLabelFormat{period}{#1~#2}
\makeatletter
\def\blfootnote{\gdef\@thefnmark{}\@footnotetext}
\makeatother

\begin{document}

\title{Koopman-based spectral clustering of \\ directed and time-evolving graphs}
\author[1]{Stefan Klus}
\author[2]{Nata\v sa Djurdjevac Conrad}

\affil[1]{School of Mathematical \& Computer Sciences, Heriot--Watt University, UK}
\affil[2]{Zuse Institute Berlin, Germany}

\date{}

\maketitle

\begin{abstract}
While spectral clustering algorithms for undirected graphs are well established and have been successfully applied to unsupervised machine learning problems ranging from image segmentation and genome sequencing to signal processing and social network analysis, clustering directed graphs remains notoriously difficult. Two of the main challenges are that the eigenvalues and eigenvectors of graph Laplacians associated with directed graphs are in general complex-valued and that there is no universally accepted definition of clusters in directed graphs. We first exploit relationships between the graph Laplacian and transfer operators and in particular between clusters in undirected graphs and metastable sets in stochastic dynamical systems and then use a generalization of the notion of metastability to derive clustering algorithms for directed and time-evolving graphs. The resulting clusters can be interpreted as \emph{coherent sets}, which play an important role in the analysis of transport and mixing processes in fluid flows.
\end{abstract}

\section{Introduction}

Spectral clustering is one of the most popular clustering techniques and---despite its sim\-plicity---often outperforms traditional clustering algorithms \cite{Luxburg07}. The goal is to identify groups of vertices in a graph that are highly connected to other vertices within the cluster, but only loosely coupled to other clusters. This can be viewed as an unsupervised learning problem. The graph could, for instance, represent the similarity between a set of given items or data points. Each item or data point corresponds to a vertex and two vertices are connected by an edge if they are similar. The degree of similarity can be represented by the associated edge weight. By applying clustering techniques to such a similarity graph, it is possible to identify groups of items or data points that share similar properties. If the similarity measure is symmetric (i.e., object $ A $ is similar to object $ B $ implies that object $ B $ is similar to object $ A $ and the weights are identical), then the resulting graph is undirected. Spectral clustering algorithms for undirected graphs are well understood and have been successfully applied to a host of different applications, see \cite{Luxburg07} for a detailed introduction and overview. If the similarity measure, however, is asymmetric, this results in a directed graph. A simple example is the internet, where website $ A $ might point to website $ B $, but not vice versa. Many different spectral clustering algorithms for directed graphs have been proposed over the last decades. A compelling idea is to turn the directed graph into an undirected graph and to then leverage state-of-the-art clustering techniques for undirected graphs. Many symmetrization approaches construct---often based on intuition or empirically determined hyperparameters---symmetric graph representations using, e.g., combinations of adjacency matrices, Laplacians, (reweighted) in- and out-degree matrices, and invariant distributions. Introducing all the proposed clustering algorithms in detail would go beyond the scope of this paper, we thus refer the reader to \cite{ZHS05, HZS06, MePe07, SaPa11, MaVa13, Ruedrich19} and references therein. Instead of constructing a symmetric matrix representation, a clustering approach that is based on computing dominant eigenvalues and eigenvectors of a complex-valued but Hermitian adjacency matrix is described in \cite{CLSZ20}.

Our approach relies on established dynamical systems theory and in particular the analysis of transfer operators---e.g., the Perron--Frobenius operator or the Koopman operator---that describe the evolution of a dynamical system \cite{Ko31, LaMa94, DJ99, BMM12, KKS16}. These methods have been successfully applied to high-dimensional molecular dynamics, fluid dynamics, and quantum mechanics problems \cite{RMBSH09, KBSS18, KNP22}, but also to stock-market, EEG, and traffic data \cite{HNDLG17, MSGR20, AvMe20}. A recent review of these methods can be found in \cite{KNKWKSN18}. Transfer operator approaches have also been used to analyze undirected graphs, e.g., for partitioning power networks~\cite{RSH16}, spectral network identification~\cite{Mauroy17}, and decentralized spectral clustering~\cite{ZKS22}. We will focus in particular on directed and time-evolving graphs and relationships between graph Laplacians and transfer operators.

It is well known that spectral clustering algorithms for undirected graphs can be interpreted in terms of random walks. The goal is to find a partition of the graph with the property that a random walker stays within a cluster for a long time and rarely jumps to other clusters \cite{Luxburg07, DjBCS11, SDjCBCS14, DjPhD}. This is equivalent to the detection of metastable sets in stochastic dynamical systems. Metastability is well defined if the process is in \emph{equilibrium}, which means that it is reversible with respect to its stationary distribution. The associated transfer operators are then self-adjoint with respect to appropriately defined inner products and the eigenvalues are consequently real-valued and can be interpreted as inherent time scales. However, in many cases the process is out of equilibrium as described in \cite{KWNS18}. The system might, for instance, be time-homogeneous but non-reversible or it might be time-inhomogeneous due to time-dependent energy potentials or external forces. In the graph setting, this corresponds to directed and time-evolving networks \cite{MPL17, PGCL18, HS19}. Non-equilibrium processes are in general not reversible and the eigenvalues of  transfer operators are complex-valued \cite{EmTr05, DjCWS16}. A natural extension of the definition of metastability to non-reversible and time-inhomogeneous systems is the notion of \emph{coherence}. Coherent sets \cite{Froyland13, AlPe15, BaKo17, FrJu18, KHMN19} are regions of the state space that are only slowly---compared to other sets---dispersed by the flow and play an essential role in the analysis of complex fluid flows. For the detection of such coherent sets, we need to analyze generalized transfer operators that are related to the forward-backward dynamics of the system \cite{BaKo17}. We will directly define these operators on graphs and also use the random walk perspective to extend spectral clustering algorithms to directed and time-evolving networks. This results in a generalized Laplacian, whose eigenvalues and eigenvectors again encode information about (potentially time-dependent) clusters. The main contributions of this work are:
\begin{itemize}[wide, itemindent=\parindent, itemsep=0ex, topsep=0.5ex]
\item We show that spectral clustering of undirected graphs corresponds to computing and analyzing eigenfunctions of the Koopman operator associated with the random walk process defined on the graph. The detected patterns can be interpreted as metastable sets.
\item We apply transfer operator theory to random walks on directed and time-evolving graphs and compute eigenfunctions of an operator that is related to the forward-backward dynamics. By clustering the eigenfunctions, we obtain coherent sets.
\item We construct benchmark problems by discretizing dynamical systems such as a rotating double-well problem and the quadruple gyre and show that random walkers starting within the same cluster will on average remain in close proximity over long time scales.
\item Furthermore, we analyze a time-evolving network that describes the social interactions of high-school students and show that the detected clusters correspond to the different specializations.
\end{itemize}
Our approach provides a clear physical interpretation of clusters in directed and time-evolving graphs and a principled way to evaluate the quality of the clustering. The remainder of the paper is structured as follows: In Section~\ref{sec:Transfer operators and graphs}, we will introduce transfer operators and directed and undirected graphs. In Section~\ref{sec:The forward-backward Laplacian}, we will show how transfer operator theory can be applied to graphs and how this relates to conventional spectral clustering techniques. Furthermore, we define the forward-backward Laplacian and analyze its properties. This allows us to extend spectral clustering methods to directed and time-evolving graphs as we will show in Section~\ref{sec:Spectral clustering}. All results will be illustrated with the aid of guiding examples and benchmark problems. Open questions and future work will be discussed in Section~\ref{sec:Conclusion}.

\section{Transfer operators and graphs}
\label{sec:Transfer operators and graphs}

Our goal is to define transfer operators on graphs and to apply data-driven methods for the approximation of these operators to time-series data generated by random walkers. In this section, the required concepts and notation will be introduced.

\subsection{Transfer operator theory}

Transfer operators such as the Perron--Frobenius operator or the Koopman operator describe the evolution of probability densities or observables of a dynamical system. Eigenvalues and eigenfunctions of these operators contain important information about global properties of the underlying system.

\subsubsection{Time-homogeneous systems}

Let $ \{ \ts X_t \ts \}_{t \ge 0} $ be a time-homogeneous stochastic process defined on the state space $ \mathbb{X} \subset \R^n $ and $ p_\tau \colon \mathbb{X} \times \mathbb{X} \to \R_{\ge 0} $ the \emph{transition density function} for a fixed \emph{lag time} $ \tau $ so that for every set $ \mathbb{A} $ it holds that
\begin{equation*}
    \mathbb{P}[X_{t+\tau} \in \mathbb{A} \mid X_t = x] = \intop_\mathbb{A} p_\tau(x,y) \ts \mathrm{d}y.
\end{equation*}
That is, $ p_\tau(x, y) $ is the probability density of $ X_{t+\tau} = y $ conditioned on $ X_t = x $.

\begin{definition}[Transfer operators]
Let $ \rho \in L^1(\mathbb{X}) $ be a probability density and $ f \in L^\infty(\mathbb{X}) $ an observable of the system.
\begin{enumerate}[wide, itemindent=\parindent, itemsep=0ex, topsep=0.5ex, label=\roman*)]
\item The \emph{Perron--Frobenius operator} $ \mathcal{P}_\tau \colon L^1(\mathbb{X}) \to L^1(\mathbb{X}) $ is given by
\begin{equation*}
    \mathcal{P}_\tau \rho(x) = \intop_{\mathbb{X}} p_{\tau}(y,x) \ts \rho(y) \ts \mathrm{d}y.
\end{equation*}
\item The \emph{Koopman operator} $ \mathcal{K}_\tau \colon L^\infty(\mathbb{X}) \to L^\infty(\mathbb{X}) $ is defined by
\begin{equation*}
    \mathcal{K}_\tau f(x) = \intop_{\mathbb{X}} p_{\tau}(x,y) \ts f(y) \ts \mathrm{d}y = \mathbb{E}[f(X_{t+\tau}) \mid X_t = x].
\end{equation*}
\end{enumerate}
\end{definition}

\begin{remark}
Assuming the system admits a unique invariant density $ \pi $, let $ u(x) \in L_\pi^1(\mathbb{X}) $ be a probability density with respect to the equilibrium density, then the \emph{Perron--Frobenius operator with respect to the equilibrium density} is defined by
\begin{equation*}
    \mathcal{T}_\tau u(x) = \intop_{\mathbb{X}} \frac{\pi(y)}{\pi(x)} \ts p_{\tau}(y, x) \ts u(y) \ts \mathrm{d}y.
\end{equation*}
\end{remark}

A dynamical system is said to be reversible if the so-called \emph{detailed balance condition}
\begin{equation*}
    \pi(x) \ts p_{\tau}(x, y) = \pi(y) \ts p_{\tau}(y, x)
\end{equation*}
is fulfilled for all $ x, y \in \mathbb{X} $. Roughly speaking, this means that the stochastic process is indistinguishable from its time-reversed counterpart. If the system is reversible---this is, for example, the case for classical molecular dynamics problems---, then the eigenvalues of the associated Perron--Frobenius operator and Koopman operator are real-valued and we can compute metastable sets by applying clustering techniques to the dominant eigenfunctions of these operators \cite{KKS16, KNKWKSN18}. For non-reversible systems, we typically obtain complex eigenvalues.

\subsubsection{Time-inhomogeneous systems}

For time-inhomogeneous systems, the transition density function and thus the operators defined above explicitly depend on the starting time~$ t $. Such systems are in general not reversible~\cite{KWNS18}. To simplify the notation, we will omit the explicit time-dependence and write again, e.g., $ \mathcal{P}_\tau $ instead of $ \mathcal{P}_{t, \tau} $.

\begin{definition}[Forward-backward operator]
Let $ \mathds{1}_\mathbb{X} $ denote the indicator function on $ \mathbb{X} $ and define $ \nu = \mathcal{P}_\tau \mathds{1}_\mathbb{X} $. The \emph{forward-backward operator} $ \mathcal{F}_\tau $ is given by
\begin{equation*}
    \mathcal{F}_\tau f(x) = \intop_{\mathbb{X}} p_\tau(x, y) \frac{1}{\nu(y)} \intop_{\mathbb{X}} p_\tau(z, y) \ts f(z) \ts \mathrm{d}z \ts \mathrm{d}y.
\end{equation*}
\end{definition}

Applying clustering techniques to the dominant eigenfunctions of the operator $ \mathcal{F}_\tau $, we obtain coherent sets, see \cite{Froyland13, BaKo17, KHMN19}.

\subsubsection{Data-driven transfer operator approximation}

Popular data-driven approaches for the approximation of transfer operators include \emph{extended dynamic mode decomposition} (EDMD)~\cite{WKR15, KKS16} and its various extensions \cite{WKR15, LDBK17, KSM19}. Assume that we have data of the form $ \{ (x^{(i)}, y^{(i)}) \}_{i=1}^m $, where $ y^{(i)} = \Theta^\tau(x^{(i)}) $ and $ \Theta^\tau $ is the flow map associated with the dynamical system. In addition to the training data, EDMD requires a vector-valued function $ \phi \colon \R^n \to \R^N $, with $ \phi(x) = [\phi_1(x), \dots, \phi_N(x)]^\top $, that maps the data into a typically higher-dimensional feature space.\!\footnote{The functions $ \phi_i $ could, for instance, be monomials, radial basis functions, or indicator functions.} Given the transformed data matrices $ \Phi_x, \Phi_y \in \R^{N \times m} $, defined by
\begin{equation*}
    \Phi_x = \begin{bmatrix} \phi(x^{(1)}) & \phi(x^{(2)}) & \dots & \phi(x^{(m)}) \end{bmatrix}
    \quad \text{and} \quad
    \Phi_y = \begin{bmatrix} \phi(y^{(1)}) & \phi(y^{(2)}) & \dots & \phi(y^{(m)}) \end{bmatrix},
\end{equation*}
we can compute empirical estimates of matrix representations of transfer operators projected onto the space spanned by $ \phi $: The approximated matrix representations of the Koopman operator $ \mathcal{K}_\tau $ and the Perron--Frobenius operator $ \mathcal{P}_\tau $ are given by
\begin{equation*}
    \widehat{K}_\tau^{(m)} = C_{xx}^+ \ts C_{xy}
    \quad \text{and} \quad
    \widehat{P}_\tau^{(m)} = C_{xx}^+ \ts C_{yx},
\end{equation*}
where $ C_{xx} = \frac{1}{m} \Phi_x \ts \Phi_x^\top $ and $ C_{xy} = C_{yx}^\top = \frac{1}{m} \Phi_x \ts \Phi_y^\top $ and $ ^+ $ denotes the pseudoinverse. Analogously, we can approximate the forward-backward operator by
\begin{equation*}
    \widehat{F}_\tau^{(m)} = C_{xx}^+ \ts C_{xy} \ts C_{yy}^+ \ts C_{yx}
\end{equation*}
as shown in \cite{KHMN19}. The pseudoinverses $ C_{xx}^+ $ and $ C_{yy}^+ $ could also be replaced by regularized inverses of the form $ (C_{xx} + \varepsilon \ts I)^{-1} $ and $ (C_{yy} + \varepsilon \ts I)^{-1} $, respectively, where $ \varepsilon $ is a regularization parameter. The representation of the forward-backward operator is closely related to \emph{canonical correlation analysis} (CCA) \cite{Hotelling36, MRB01}, which aims at maximizing the correlation between multidimensional random variables, see \cite{KHMN19} for details.

\begin{remark}
For the approximation of the Perron--Frobenius operator, we need to assume that $ x^{(i)} $ is sampled from the uniform distribution. If we use one long equilibrated trajectory instead, then $ x^{(i)} \sim \pi $ and we obtain $ \widehat{T}_\tau^{(m)} = C_{xx}^+ \ts C_{yx} $, i.e., an approximation of the reweighted operator $ \mathcal{T}_\tau $, cf.~\cite{KNKWKSN18}.
\end{remark}

\subsection{Graph theory}

We assume the reader to be familiar with graph-theoretical concepts and will thus only briefly define directed and undirected graphs as well as associated matrix representations, a more detailed introduction can be found, e.g., in \cite{CLRS09, Rigo16, LaSch22}.

\subsubsection{Basic graph properties}

A \emph{directed graph} $ \mc{G} = (\mc{V}, \mc{E}) $ is given by a set of vertices $ \mc{V} = \{ \mc[1]{v}, \dots, \mc[n]{v} \} $ and a set of edges $ \mc{E} \subseteq \mc{V} \times \mc{V} $. An \emph{undirected graph} can then be regarded as a special case, where the edges have no directionality.

\begin{definition}[Weighted adjacency matrix]
The \emph{weighted adjacency matrix} $ A = (a_{ij})_{i,j=1}^n $ associated with a graph $ \mc{G} $ is defined by
\begin{equation*}
    a_{ij} =
    \begin{cases}
        w(\mc[i]{v}, \mc[j]{v}), & \text{if } (\mc[i]{v}, \mc[j]{v}) \in \mc{E}, \\
        0, & \text{otherwise},
    \end{cases}
\end{equation*}
where $ w \colon \mc{V} \times \mc{V} \to \R_{\ge 0} $ is a function that determines the weight of the edge $ (\mc[i]{v}, \mc[j]{v}) $.
\end{definition}

Furthermore, we define the \emph{out-degree} $ \mathscr{o}(\mc[i]{v}) $ of a vertex $ \mc[i]{v} $ and the degree matrix $ D_\mathscr{o} $ by
\begin{equation*}
    \mathscr{o}(\mc[i]{v}) = \sum_{j=1}^n a_{ij}
    \quad \text{and} \quad
    D_\mathscr{o} = \diag\big(\mathscr{o}(\mc[1]{v}), \dots, \mathscr{o}(\mc[n]{v})\big).
\end{equation*}

In many applications, the network structure is changing in time. We will consider time-evolving graphs---also called temporal graphs, dynamic graphs, or time-varying graphs---where the number of vertices is fixed, but the number of edges can change. Different approaches have been developed to study dynamics of network change. Aggregation-based methods merge all temporal information into one graph \cite{OSHSLKKB07}, focusing on the global network structure and neglecting information on smaller scales. More temporal information is preserved when aggregating networks to, e.g., multilayer networks \cite{BBCGGRSWZ14}. In this paper, we  focus on approaches where each time-slice is analyzed separately and then this information is studied to extract global network change \cite{HS19}. In particular, we assume that a temporal change of a graph at discrete points in time, $ t_1, t_2, \dots $, is formally given by a sequence $(A^{(t_1)}, A^{(t_2)},\dots)$, where  $A^{(t_i)}$ is the weighted adjacency matrix of the graph at time $ t_i $.

\subsubsection{Random walks on graphs}

A time-discrete random walk on a graph $ \mc{G} $ is a discrete stochastic process $ X_t $ that starts in a vertex $ \mc[i]{v} $ and at each time step moves to an adjacent vertex $ \mc[j]{v} $ (or stays in $ \mc[i]{v} $ if self-loops are allowed) with a probability that is proportional to the weight of the edge $ (\mc[i]{v}, \mc[j]{v}) $. Such a random walk on $ \mc{G} $ is defined by the row-stochastic \emph{transition probability matrix}
\begin{equation} \label{eq:RW}
   P = D_\mathscr{o}^{-1} A.
\end{equation}
That is, the entry $ p_{ij} $ of $ P $ is given by $ p_{ij} = \mathbb{P}[X_{t+1} = \mc[j]{v} \mid X_t = \mc[i]{v}\,] $. If a directed graph $ \mc{G} $ is strongly connected and aperiodic,\!\footnote{A graph is called aperiodic if the greatest common divisor of the lengths of its cycles is equal to $ 1 $.} then the random walk process is ergodic and it converges to a unique stationary distribution. If $ \mc{G} $ does not satisfy these properties, typically the random walk process is modified by including a so-called teleportation probability, i.e., a small probability to randomly jump to any vertex in the graph \cite{BrLa98}. If the graph is undirected, the adjacency matrix is symmetric and the matrix $ P $ is similar to a symmetric matrix and its spectrum real-valued. Many variants of Laplacians for directed and undirected graphs have been used in the literature, we will adopt the following definition, cf.~\cite{Luxburg07}.

\begin{definition}[Graph Laplacian]
The \emph{unnormalized graph Laplacian} associated with $ \mc{G} $ is defined by $ L = D_\mathscr{o} - A $ and the \emph{random-walk normalized graph Laplacian} by
\begin{equation*}
    L_\textup{rw} = D_\mathscr{o}^{-1} \ts L = I - P.
\end{equation*}
\end{definition}

Note that if $ \lambda $ is an eigenvalue of $ P $, then $ 1 - \lambda $ is an eigenvalue of $ L_\textup{rw} $. Furthermore, the eigenvectors of the two matrices are identical \cite{MS01}.

\subsubsection{Spectral clustering of undirected graphs}

Several different clustering algorithms based on graph Laplacians have been proposed, we will use the normalized spectral clustering algorithm as defined in \cite{Luxburg07}.

\begin{mdframed}[backgroundcolor=boxback,hidealllines=true]
\vspace*{-1ex}
\begin{textalgorithm}[Spectral clustering algorithm for undirected graphs] \label{alg:undirected} $ $
\begin{enumerate}
\item Assemble the random-walk Laplacian $ L_\textup{rw} = I - P $.
\item Compute the $ k $ smallest eigenvalues $ \lambda_\ell $ and associated eigenvectors $ u_\ell $ of $ L_\textup{rw} $.
\item Define $ U = [u_1, \dots, u_k] \in \R^{n \times k} $ and let $ r_i $ denote the $i$th row of $ U $.
\item Cluster the points $ \{ \ts r_i \ts \}_{i=1}^n $ using, e.g., $ k $-means.
\end{enumerate}
\end{textalgorithm}
\end{mdframed}

We can either use $ L_\textup{rw} $ or $ P $ for spectral clustering. The only difference is that for the former we have to compute the smallest and for the latter the largest eigenvalues. The number of clusters $ k $ is typically chosen in such a way that there exists a spectral gap between $ \lambda_k $ and $ \lambda_{k+1} $ \cite{DjBCS11, DjPhD}. We can also apply other clustering algorithms in step~4. In what follows, we will sometimes use the \emph{sparse eigenbasis approximation} (SEBA) algorithm~\cite{FRS19}, which is advantageous when there is no clear eigengap and was specifically developed for the extraction of metastable and coherent sets.

\section{The forward-backward Laplacian}
\label{sec:The forward-backward Laplacian}

We will now define transfer operators on graphs and highlight relationships with conventional spectral clustering algorithms.

\subsection{Transfer operator perspective}

We consider the case $ \tau = 1 $. That is, given a discrete distribution $ \rho $ on the graph at time $ t $ the Perron--Frobenius operator applied to $ \rho $ yields the distribution at time $ t + 1 $. Correspondingly, the Koopman operator describes the evolution of observables $ f $. From a data-driven perspective, this means that each random walker takes just one step. The state space of the random walkers is given by $ \mathbb{X} = \mc{V} = \{\mc[1]{v}, \dots, \mc[n]{v} \} $ and since $ \tau = 1 $ it holds that $ p_\tau(\mc[i]{v}, \mc[j]{v}) = p_{ij} $. The Perron--Frobenius operator and the Koopman operator defined on a graph $ \mc{G} $ can thus be written as
\begin{equation*}
    \mathcal{P}_\tau \rho(\mc[i]{v}) = \sum_{j=1}^n p_{ji} \ts \rho(\mc[j]{v})
    \quad \text{and} \quad
    \mathcal{K}_\tau f(\mc[i]{v}) = \sum_{j=1}^n p_{ij} \ts f(\mc[j]{v}),
\end{equation*}
respectively. Analogously, the forward-backward operator $ \mathcal{F}_\tau $ can be expressed as
\begin{equation*}
    \mathcal{F}_\tau f(\mc[i]{v}) = \sum_{j=1}^n p_{ij} \frac{1}{\nu(\mc[j]{v})} \sum_{k=1}^n p_{kj} \ts f(\mc[k]{v}),
    \quad \text{with} \quad
    \nu(\mc[j]{v}) = \sum_{\ell=1}^n p_{\ell j}.
\end{equation*}
With a slight abuse of notation, we define vectors $ \rho, f \in \R^n $ with entries $ \rho_i = \rho(\mc[i]{v}) $ and $ f_i = f(\mc[i]{v}) $ so that we can write
\begin{equation*}
    \mathcal{P}_\tau \rho = P^\top \rho
    \quad \text{and} \quad
    \mathcal{K}_\tau f = P f,
\end{equation*}
where the operators are applied component-wise. The matrix representation of the operator $ \mathcal{F}_\tau $ is given by
\begin{equation*}
    \mathcal{F}_\tau f = P D_\nu^{-1} P^\top f \stackrel{\eqref{eq:RW}}{=} D_\mathscr{o}^{-1} A D_\nu^{-1} A^\top D_\mathscr{o}^{-1} f =: Q \ts f,
\end{equation*}
where $ D_\nu = \diag\big(\nu(\mc[1]{v}), \dots, \nu(\mc[n]{v})\big) $. To ensure that the inverse of $ D_\nu $ exists, the probability that a random walker (starting in any vertex) ends up in $ \mc[j]{v} $ after one step must be nonzero, i.e., there must be at least one incoming edge. We account for this by including self-loops with a small weight to each vertex, which can be regarded as a form of regularization. Another possible approach would be to introduce teleportation probabilities, as described in Section \ref{sec:Transfer operators and graphs}. 

\begin{lemma}
It holds that $ Q = P D_\nu^{-1} P^\top $ is a doubly stochastic matrix.
\end{lemma}

\begin{proof}
First, note that $ P $ and $ D_\nu^{-1} P^\top $ are row-stochastic matrices and thus their product. Furthermore, $ Q $ is symmetric.
\end{proof}

In order to determine eigenfunctions of the operators introduced above, we can thus simply compute eigenvectors of the corresponding matrix representations. This illustrates that the conventional  spectral clustering for undirected graphs is based on the eigenfunctions of the Koopman operator. For directed and time-evolving graphs, we will now extend this idea and propose a clustering algorithm that utilizes eigenfunctions of the forward-backward operator.

\begin{definition}[Forward-backward Laplacian]
We call the matrix
\begin{equation*}
    L_\textup{fb} = I - Q
\end{equation*}
associated with the graph $ \mc{G} $ the \emph{forward-backward Laplacian}.
\end{definition}

That is, we define the forward-backward Laplacian based on the generalized operator $ \mathcal{F}_\tau $ by mirroring the definition of the random-walk Laplacian as a shifted matrix representation of the Koopman operator $ \mathcal{K}_\tau $.

\subsection{Random walker perspective}

We derived the forward-backward Laplacian by applying the definition of transfer operators to graphs. Alternatively, we can interpret these results in terms of random walks again. We define $ \phi(x) = [\phi_1(x), \dots, \phi_n(x)]^\top $, with
\begin{equation*}
    \phi_i(x) =
    \begin{cases}
        1, & x = \mc[i]{v}, \\
        0, & \text{otherwise}.
    \end{cases}
\end{equation*}
That is, the basis function $ \phi_i(x) $ is the indicator function for vertex $ \mc[i]{v} $ and the random walk data is simply represented in the one-hot encoding format. Furthermore, the feature space dimension is $ N = n $.

\begin{proposition} \label{prop:Convergence}
Assume $ x^{(i)} \sim U(\{1, \dots, n\}) $, where $ U $ denotes the uniform distribution. Using a basis comprising indicator functions, we obtain
\begin{align*}
    & \lim\limits_{m \to \infty}{\widehat{K}_\tau^{(m)}} = P,
    && \lim\limits_{m \to \infty}{\widehat{P}_\tau^{(m)}} = P^\top,
    && \lim\limits_{m \to \infty}{\widehat{F}_\tau^{(m)}} = Q.
\end{align*}
\end{proposition}

The proof is an application of EDMD and CCA convergence results to discrete Markov chains, where the dictionary is now given by indicator functions. For the sake of completeness, it is included in the appendix.

\begin{remark}
We can also approximate the Koopman operator associated with a time-homogeneous system with invariant density $ \pi $ using one long equilibrated trajectory. Furthermore, in this case $ \widehat{T}_\tau^{(m)} = C_{xx}^+ \ts C_{yx} $ converges to the Perron--Frobenius operator with respect to the equilibrium density.
\end{remark}

This shows that we can estimate metastable and coherent sets from data and also allows us to interpret the spectral clustering methods for directed and time-evolving graphs in terms of random walkers.

\section{Spectral clustering of directed and time-evolving graphs}
\label{sec:Spectral clustering}

We will now illustrate how the forward-backward Laplacian $ L_\textup{fb} $ can be used for spectral clustering.

\subsection{Spectral clustering of directed graphs}

We have seen that the detection of clusters in undirected graphs is related to the computation of metastable sets. In order to detect clusters in directed graphs, we compute coherent sets.

\begin{mdframed}[backgroundcolor=boxback,hidealllines=true]
\vspace*{-1ex}
\begin{textalgorithm}[Spectral clustering algorithm for directed graphs] \label{alg:directed} $ $
\begin{enumerate}
\item Assemble the forward-backward Laplacian $ L_\textup{fb} = I - Q $.
\item Compute the $ k $ smallest eigenvalues $ \lambda_\ell $ and associated eigenvectors $ u_\ell $ of $ L_\textup{fb} $.
\item Define $ U = [u_1, \dots, u_k] \in \R^{n \times k} $ and let $ r_i $ denote the $i$th row of $ U $.
\item Cluster the points $ \{ \ts r_i \ts \}_{i=1}^n $ using, e.g., $ k $-means.
\end{enumerate}
\end{textalgorithm}
\end{mdframed}

As before, instead of computing the $ k $ smallest eigenvalues of $ L_\textup{fb} $, we can determine the $ k $ largest eigenvalues and associated eigenvectors of $ Q $. Alternatively, we can compute the first $ k $ right singular vectors of the matrix $ D_\nu^{-\nicefrac{1}{2}} P^\top = D_\nu^{-\nicefrac{1}{2}} A^\top D_\mathscr{o}^{-1} $. This shows that spectral clustering for directed graphs can again be regarded as a spectral decomposition of an appropriately normalized adjacency matrix.

\begin{example} \label{ex:Guiding example}

\begin{figure}
    \centering
    \begin{minipage}[t]{0.25\textwidth}
        \centering
        \subfiguretitle{(a)}
        \vspace*{1ex}
        \resizebox{0.85\textwidth}{!}{%
        \begin{tikzpicture}[
                >= stealth, 
                semithick 
            ]
            \tikzstyle{every state}=[
                draw = black,
                thick,
                fill = white,
                inner sep=0pt,
                text width=6mm,
                align=center,
                scale=0.6
            ]
            
            \node[state] (v3) {3};
            \node[state] (v4) [right=0.8cm of v3] {4};
            \node[state] (v1) [below=0.8cm of v4] {1};
            \node[state] (v2) [below=0.8cm of v3] {2};
            
            \node[state] (v8) [right=1.35cm of v4] {8};
            \node[state] (v5) [above left=0.55cm and 0.55cm of v8] {5};
            \node[state] (v6) [above=1.35cm of v8] {6};
            \node[state] (v7) [right=1.3cm of v5] {7};
            
            \node[state] (v9) [right=1.35cm of v1] {9};
            \node[state] (v10) [below right=0.55cm and 0.55cm of v9] {10};
            \node[state] (v11) [below=1.35cm of v9] {11};
            \node[state] (v12) [left=1.35cm of v10] {12};
            
            \path[->] (v1) edge node {} (v2);
            \path[->] (v2) edge node {} (v3);
            \path[->] (v3) edge node {} (v4);
            \path[->] (v4) edge node {} (v1);
            
            \path[->] (v5.45) edge node {} (v6.225);
            \path[->] (v6.-45) edge node {} (v7.135);
            \path[->] (v7.-135) edge node {} (v8.45);
            \path[->] (v8.135) edge node {} (v5.-45);
            
            \path[->] (v9.-45) edge node {} (v10.135);
            \path[->] (v10.-135) edge node {} (v11.45);
            \path[->] (v11.135) edge node {} (v12.-45);
            \path[->] (v12.45) edge node {} (v9.225);
            
            \path[->, dashed] (v4) edge node [dotted] {} (v5);
            \path[->, dashed] (v8) edge node [dotted] {} (v9);
            \path[->, dashed] (v12) edge node [dotted] {} (v1);
        \end{tikzpicture}}
    \end{minipage}
    \begin{minipage}[t]{0.268\textwidth}
        \centering
        \subfiguretitle{(b)}
        \vspace*{1ex}
        \includegraphics[width=\linewidth]{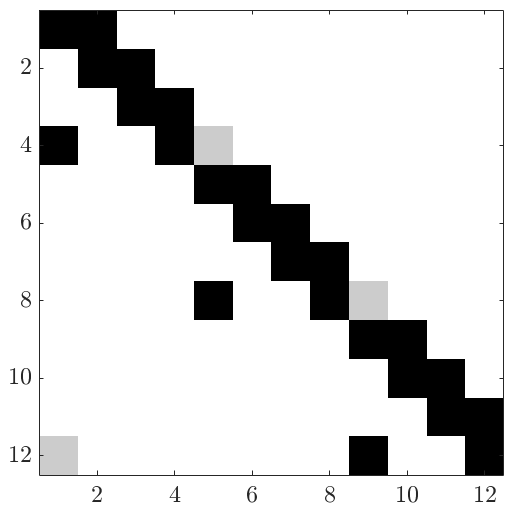}
    \end{minipage}
    \begin{minipage}[t]{0.4\textwidth}
        \centering
        \subfiguretitle{(c)}
        \vspace*{1ex}
        \includegraphics[width=\linewidth]{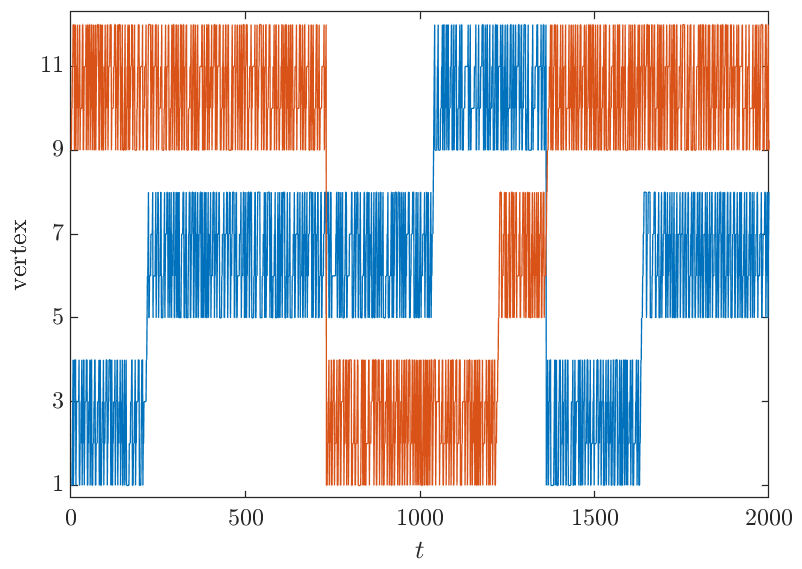}
    \end{minipage} \\[0.5ex]
    \begin{minipage}[t]{0.302\textwidth}
        \centering
        \subfiguretitle{(d)}
        \vspace*{1ex}
        \includegraphics[width=\linewidth]{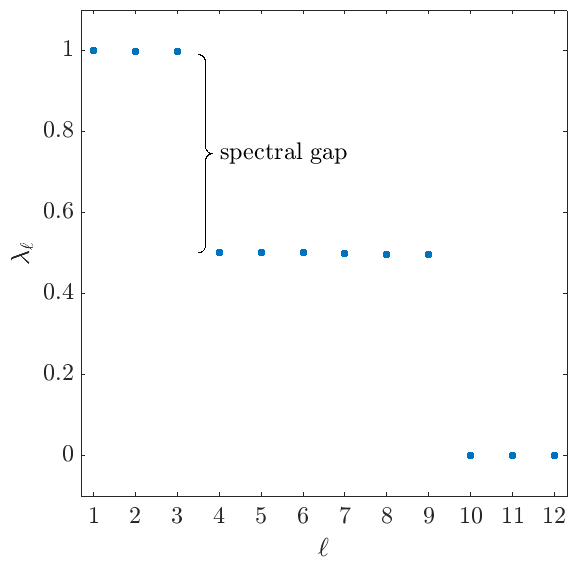}
    \end{minipage}
    \begin{minipage}[t]{0.4\textwidth}
        \centering
        \subfiguretitle{(e)}
        \vspace*{1ex}
        \includegraphics[width=\linewidth]{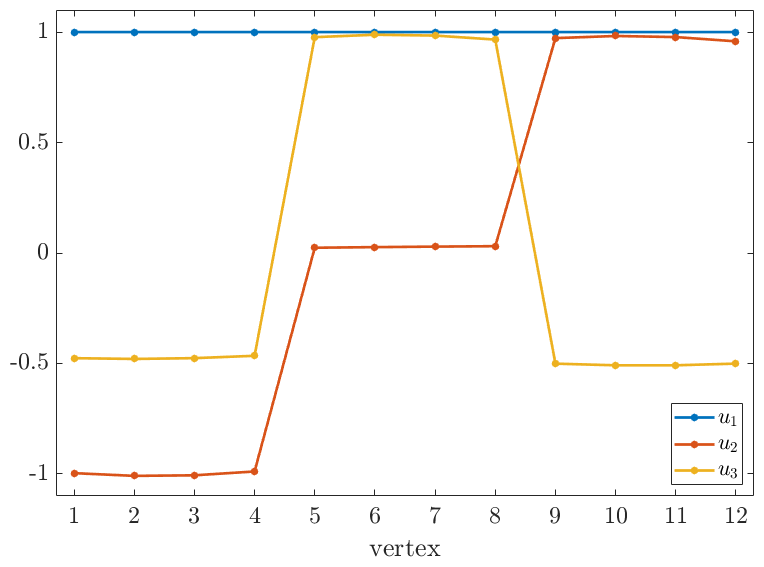}
    \end{minipage}
    \begin{minipage}[t]{0.25\textwidth}
        \centering
        \subfiguretitle{(f)}
        \vspace*{1ex}
        \resizebox{0.85\textwidth}{!}{%
        \begin{tikzpicture}[
                >= stealth, 
                semithick 
            ]
            \tikzstyle{every state}=[
                draw = black,
                thick,
                fill = white,
                inner sep=0pt,
                text width=6mm,
                align=center,
                scale=0.6
            ]
            
            \node[state,fill=red!60] (v3) {3};
            \node[state,fill=red!60] (v4) [right=0.8cm of v3] {4};
            \node[state,fill=red!60] (v1) [below=0.8cm of v4] {1};
            \node[state,fill=red!60] (v2) [below=0.8cm of v3] {2};
            
            \node[state,fill=green!60] (v8) [right=1.35cm of v4] {8};
            \node[state,fill=green!60] (v5) [above left=0.55cm and 0.55cm of v8] {5};
            \node[state,fill=green!60] (v6) [above=1.35cm of v8] {6};
            \node[state,fill=green!60] (v7) [right=1.3cm of v5] {7};
            
            \node[state,fill=yellow!60] (v9) [right=1.35cm of v1] {9};
            \node[state,fill=yellow!60] (v10) [below right=0.55cm and 0.55cm of v9] {10};
            \node[state,fill=yellow!60] (v11) [below=1.35cm of v9] {11};
            \node[state,fill=yellow!60] (v12) [left=1.35cm of v10] {12};
            
            \path[->] (v1) edge node {} (v2);
            \path[->] (v2) edge node {} (v3);
            \path[->] (v3) edge node {} (v4);
            \path[->] (v4) edge node {} (v1);
            
            \path[->] (v5.45) edge node {} (v6.225);
            \path[->] (v6.-45) edge node {} (v7.135);
            \path[->] (v7.-135) edge node {} (v8.45);
            \path[->] (v8.135) edge node {} (v5.-45);
            
            \path[->] (v9.-45) edge node {} (v10.135);
            \path[->] (v10.-135) edge node {} (v11.45);
            \path[->] (v11.135) edge node {} (v12.-45);
            \path[->] (v12.45) edge node {} (v9.225);
            
            \path[->, dashed] (v4) edge node [dotted] {} (v5);
            \path[->, dashed] (v8) edge node [dotted] {} (v9);
            \path[->, dashed] (v12) edge node [dotted] {} (v1);
        \end{tikzpicture}}
    \end{minipage}
    \caption{(a) Directed graph with three clusters. The weight of each solid edge is 1 and the weight of each dashed edge 0.01. Self-loops are omitted. (b) Corresponding asymmetric adjacency matrix. (c) Two random walks  of length 2000 starting in different vertices. (d)~Eigenvalues of the matrix $ Q $. Three eigenvalues are close to 1, which implies that there are three coherent sets. (e)~Eigenvectors corresponding to the dominant eigenvalues. By applying $ k $-means, SEBA, or other clustering techniques to the eigenvectors, we can extract the coherent sets. (f) Resulting clustering of the graph.}
    \label{fig:Guiding example}
\end{figure}
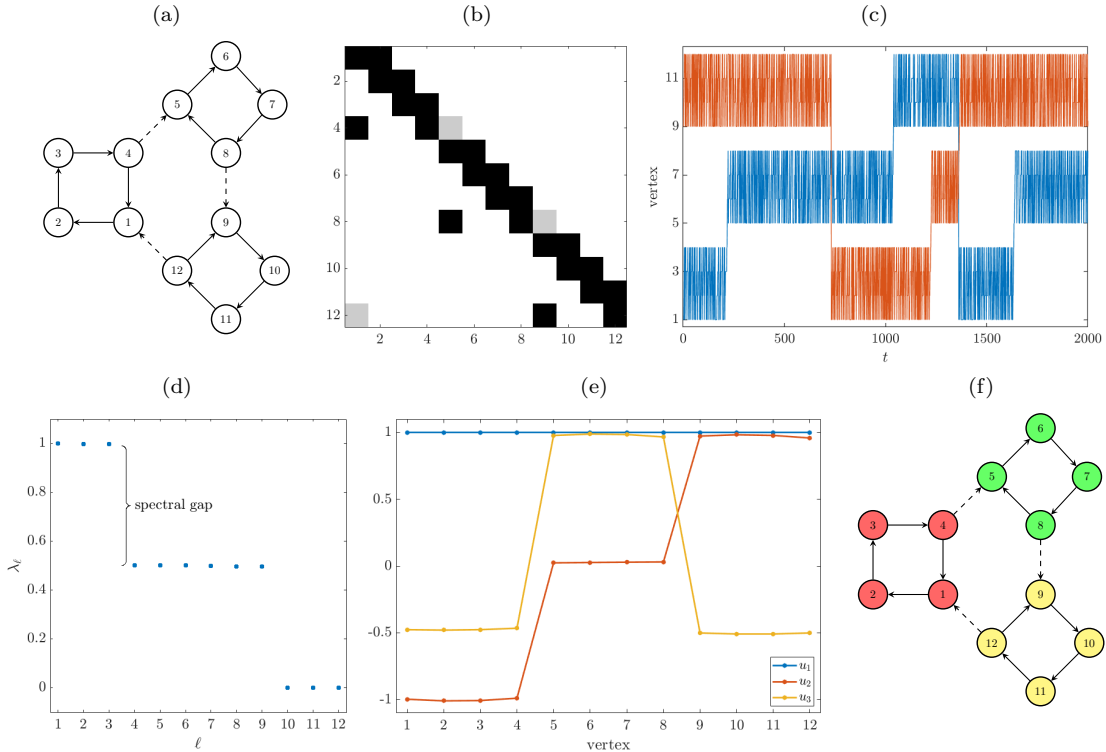

Consider the directed graph shown in Figure~\ref{fig:Guiding example}. The graph comprises three unidirectionally connected clusters and a random walker will typically spend a long time in one cluster before moving to the next one. Although this behavior is highly similar to the undirected case, the eigenvalues and eigenvectors of $ L_\textup{rw} $ are complex-valued and standard spectral clustering techniques fail. We instead compute the eigenvalues and eigenvectors of the forward-backward Laplacian $ L_\textup{fb} $ and then apply clustering techniques to the dominant eigenvectors. Due to the symmetry, we obtain repeated eigenvalues and the corresponding eigenvectors are only determined up to basis rotations. Nevertheless, it can be seen that the values of the eigenvectors are almost constant within the clusters. This indicates a crisp clustering. \exampleSymbol
\end{example}

In order to avoid having to deal with complex-valued eigenvalues and eigenvectors, one might be tempted to consider only the real parts (or a combination of real and imaginary parts). This will, however, in general not lead to satisfactory results. Even for the simple graph introduced in Example~\ref{ex:Guiding example}, we would not obtain the three clusters shown in Figure~\ref{fig:Guiding example}. It is important to note here that the self-loops are crucial and improve the clustering results. Without the self-loops, a random walker starting in vertex 1 that moves forward and then backward will end up in vertex 1 with probability one. This leads to clusters of size one. Adding self-loops hence regularizes the problem and leads to more balanced clusters. Unless noted otherwise, we will always add self-loops with edge weights $ w(\mc[i]{v}, \mc[i]{v}) = 1 $ to all vertices. Let us now apply Algorithm~\ref{alg:directed} to a larger graph that does not have such a clearly defined cluster structure.

\begin{example}

\begin{figure}[htb]
    \centering
    \begin{minipage}[t]{0.48\textwidth}
        \centering
        \subfiguretitle{(a)}
        \includegraphics[width=0.85\textwidth]{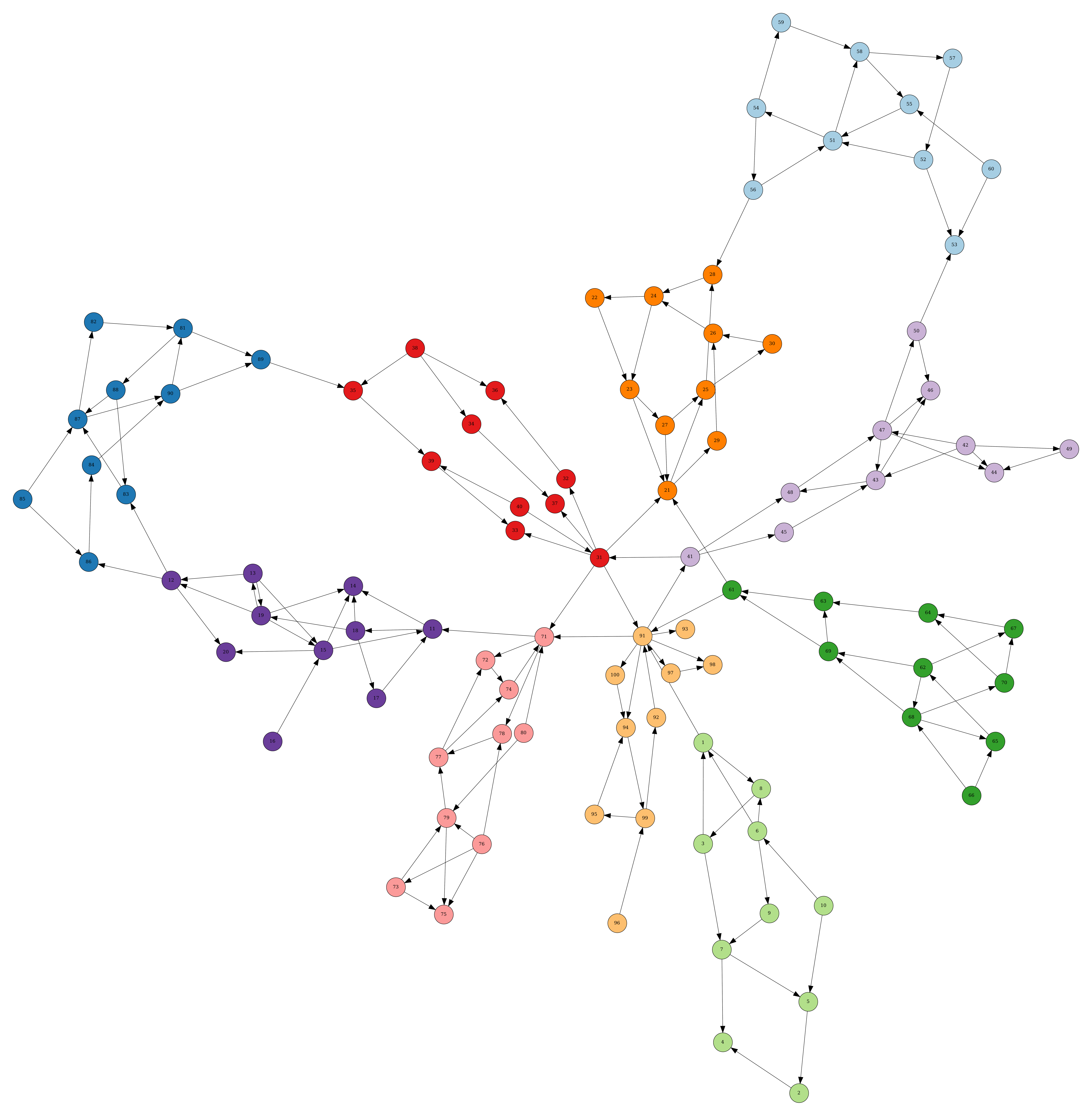}
    \end{minipage}
    \begin{minipage}[t]{0.48\textwidth}
        \centering
        \subfiguretitle{(b)}
        \vspace*{3ex}
        \includegraphics[width=0.85\textwidth]{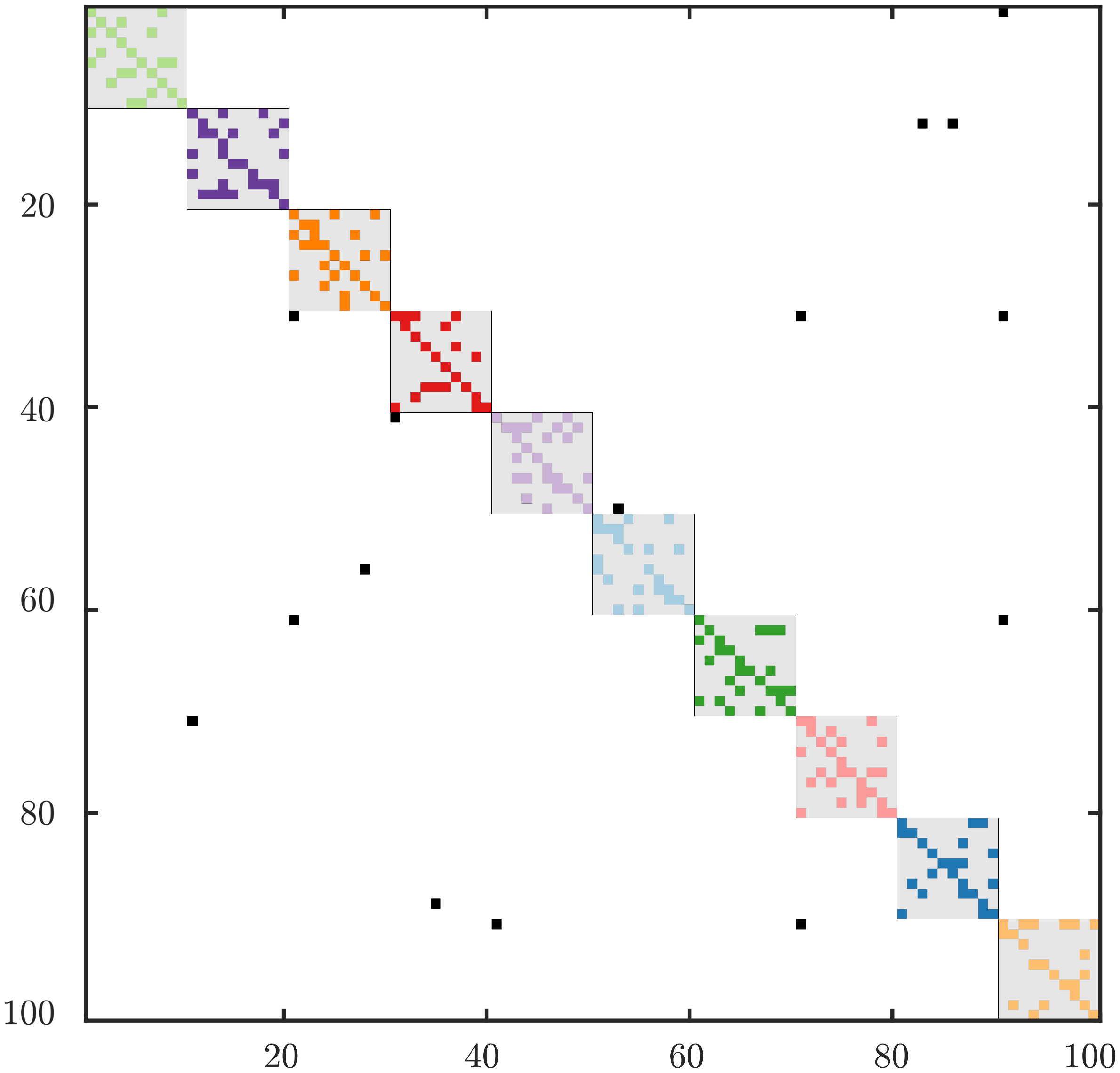}
    \end{minipage}
    \caption{(a) Randomly generated directed graph with 100 vertices consisting of 10 sparsely connected graphs with 10 vertices. We apply $ k $-means with $ k = 10 $ to the dominant eigenvectors of $ L_\textup{fb} $. The resulting clusters are represented in different colors. (b)~Adjacency matrix of the graph, where the clusters are marked in the corresponding colors.}
    \label{fig:Ten-cluster graph}
\end{figure}

We generate a directed graph comprising 100 vertices by sparsely connecting 10 randomly generated sparse matrices of size 10. The clustered graph and its adjacency matrix are shown in Figure~\ref{fig:Ten-cluster graph}. Algorithm~\ref{alg:directed} splits the graph into the 10 clusters associated with the 10 randomly generated matrices of size 10. \exampleSymbol
\end{example}

This shows that the proposed spectral clustering algorithm for directed graphs successfully identifies groups of vertices that share similar properties. In particular, reversing the direction of certain edges may also change the cluster assignments, which illustrates that the directionality of the edges is taken into account. The proposed spectral clustering algorithm requires only matrix-matrix multiplications, inverses of diagonal matrices, and methods to compute spectral properties of the resulting forward-backward Laplacian. Using state-of-the-art numerical linear algebra libraries containing, for instance, iterative Arnoldi-type methods for the computation of dominant eigenvalues and eigenvectors of high-dimensional sparse matrices, the algorithm can be easily applied to large-scale problems. 

\begin{example}

\begin{figure}[t]
    \centering
    \begin{minipage}[t]{0.35\textwidth}
        \centering
        \subfiguretitle{(a)}
        \vspace*{1.5ex}
        \includegraphics[width=0.95\linewidth]{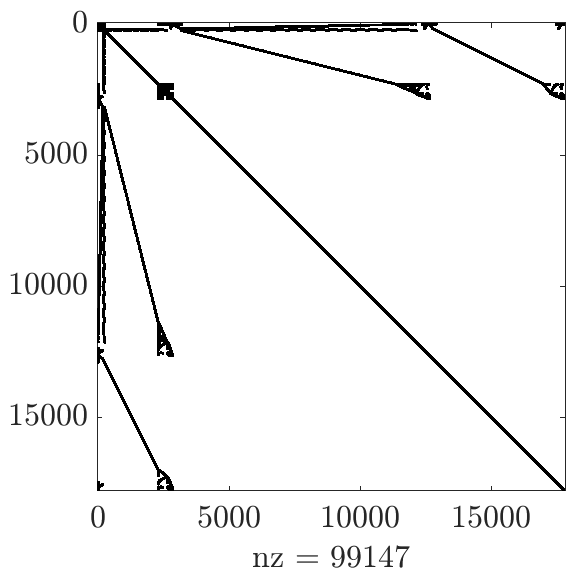}
    \end{minipage}
    \begin{minipage}[t]{0.62\textwidth}
        \centering
        \subfiguretitle{(b)}
        \includegraphics[width=\linewidth]{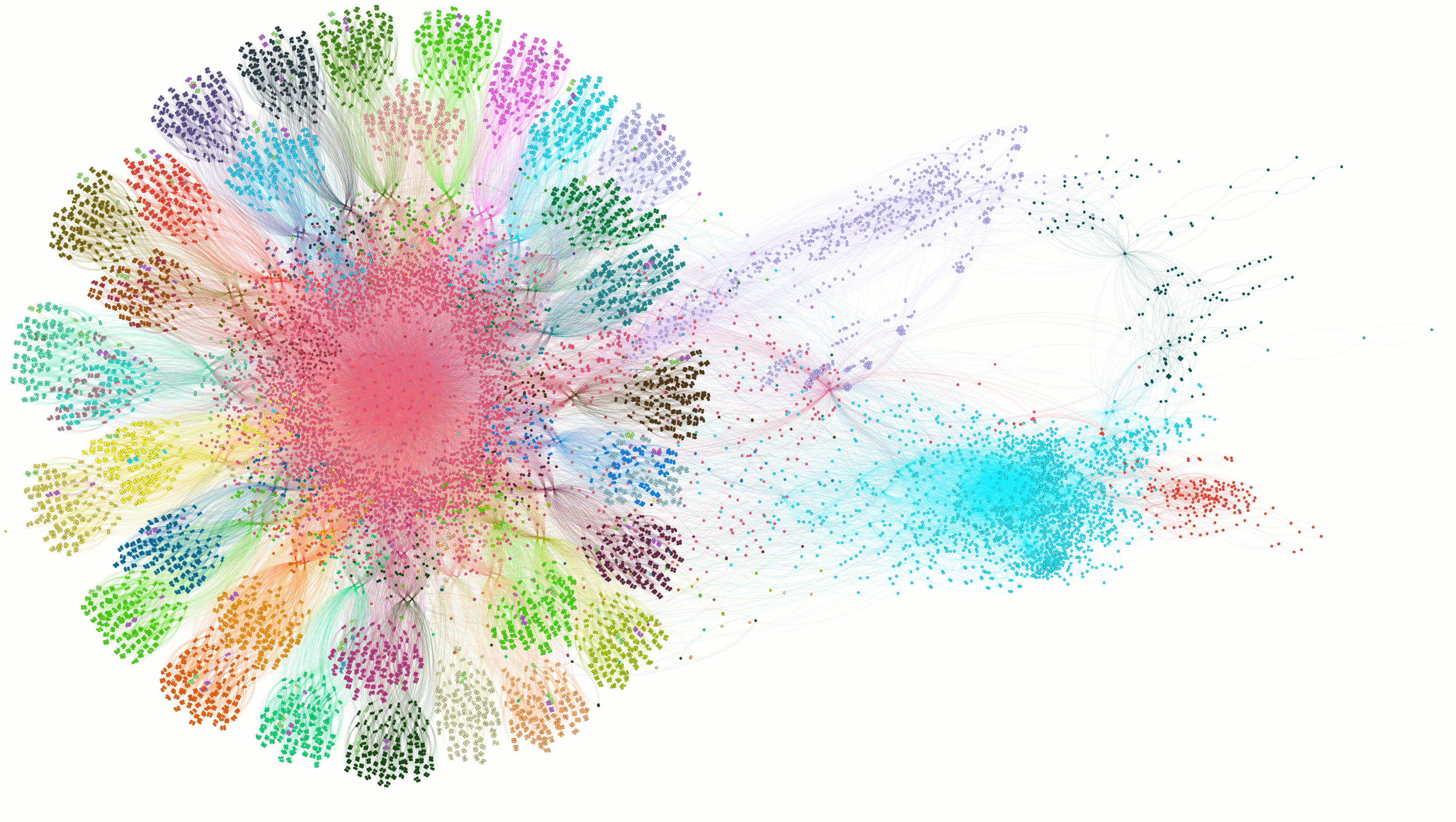}
    \end{minipage}
    \caption{(a) Adjacency matrix of the memory circuit. (b) Spectral clustering of the graph into 50 clusters. In addition to the pink cluster in the middle, which contains approximately 18 \% of the vertices, and the light-blue cluster, which contains approximately 9 \%, we obtain many regular-looking clusters surrounding the center cluster.}
    \label{fig:Matrix market}
\end{figure}

We cluster a directed graph representing a memory circuit. The matrix, which is available on the \href{https://math.nist.gov/MatrixMarket/data/misc/hamm/memplus.html}{Matrix Market} website, is of size $ 17758 \times 17758 $ and contains $ 99147 $ nonzero entries. We apply Algorithm~\ref{alg:directed} and arbitrarily choose $ k = 50 $ since there is no clear spectral gap in this case. Applying the clustering algorithm just takes a couple of seconds on a conventional laptop. The clustering is shown in Figure~\ref{fig:Matrix market}. The results demonstrate the efficacy and scalability of our approach. \exampleSymbol
\end{example}

\subsection{Spectral clustering of time-evolving graphs}

To illustrate the versatility of the forward-backward Laplacian, we will now apply our approach to find coherent sets in time-evolving graphs. In particular, we will consider the following two approaches for estimating the matrix $ Q $: In Approach A, the matrix $ Q $ is obtained from random walk data by computing $ \widehat{F}_\tau^{(m)} = C_{xx}^+ \ts C_{xy} \ts C_{yy}^+ \ts C_{yx} $, see Proposition~\ref{prop:Convergence}. In Approach B, we define $ Q = P D_\nu^{-1} P^\top $ using
\begin{equation*}
    P = \prod_{t=0}^{T} P^{(t)},
\end{equation*}
where the transition matrix $ P^{(t)} $ corresponds to the adjacency matrix of the graph at time~$ t \in \{0, \dots, T \} $. Constructed in this way, the matrix $ P $ contains transition probabilities over time $ T $. We then apply again Algorithm~\ref{alg:directed} to obtain the
clusters. The two approaches differ in that the former requires only random walk data, whereas the latter assumes that the time-evolving network structure is known. In the following examples, we will demonstrate how both approaches can be applied to real-world data. We will use Approach~A in Example \ref{ex:Double-well graph}, mainly to illustrate the notion of coherent sets in time-evolving graphs, and Approach~B in Example \ref{ex:Quadruple-gyre graph} and Example \ref{ex:School network} to show how coherent sets can be found when the structure of a time-evolving graph is known.

\begin{example} \label{ex:Double-well graph}

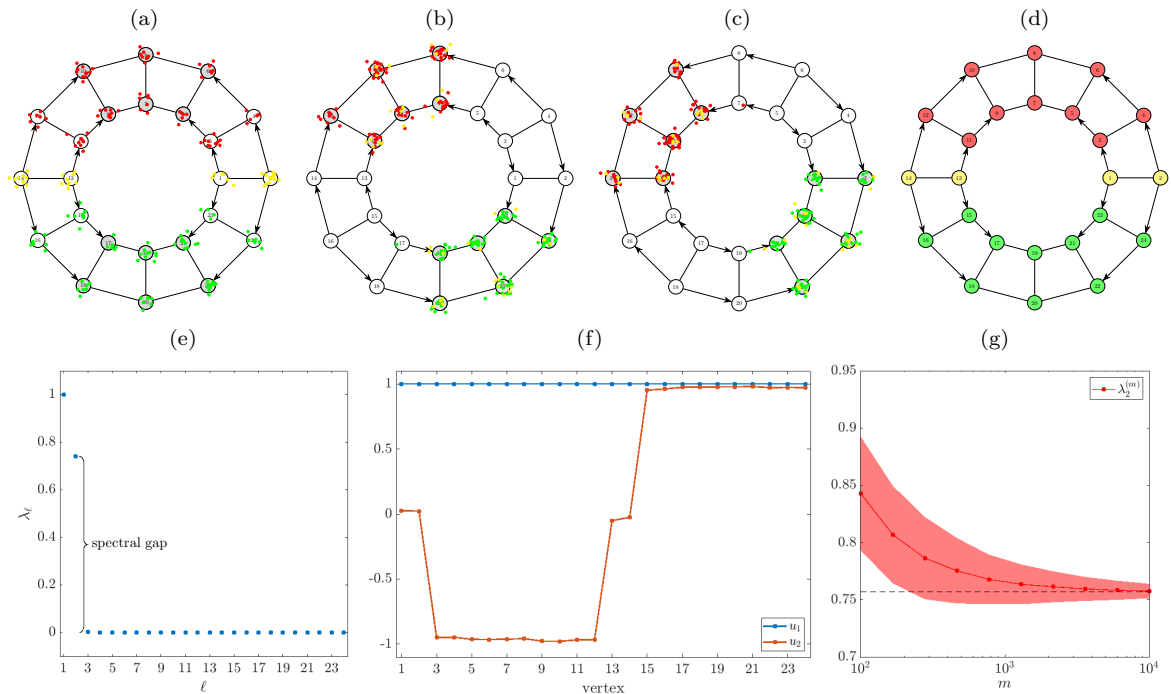
\begin{figure}
    \centering
    \begin{minipage}[t]{0.235\textwidth}
        \centering
        \subfiguretitle{(a)}
        \vspace*{1ex}
        \resizebox{\textwidth}{!}{%
        \begin{tikzpicture}[
                        >= stealth, 
                        semithick 
                    ]
                    \tikzstyle{every state}=[
                        draw = black,
                        thick,
                        fill = white,
                        inner sep=0pt,
                        text width=6mm,
                        align=center,
                        scale=0.6
                    ]
          \foreach \x in {1,3,...,24}
            \node[state] (\x) at (\x*360/24-15:3) {\x};
          \foreach \x in {2,4,...,24}
            \node[state] (\x) at (\x*360/24-30:5) {\x};
          \foreach \x/\y in {1/2,3/4,5/6,7/8,9/10,11/12,13/14,15/16,17/18,19/20,21/22,23/24, 6/8,8/10,5/7,7/9,18/20,20/22,17/19,19/21}
            \draw [-] (\x) -- (\y);
          \foreach \x/\y in {2/4,4/6,1/3,3/5,2/24,24/22,1/23,23/21,14/12,12/10,13/11,11/9,14/16,16/18,13/15,15/17}
            \draw [-{Stealth[length=3mm, width=2mm]}] (\x) -- (\y);
          
          \node[state,fill=gray!30] () at (5*360/24-15:3) {5};
          \node[state,fill=gray!30] () at (7*360/24-15:3) {7};
          \node[state,fill=gray!30] () at (9*360/24-15:3) {9};
          \node[state,fill=gray!30] () at (6*360/24-30:5) {6};
          \node[state,fill=gray!30] () at (8*360/24-30:5) {8};
          \node[state,fill=gray!30] () at (10*360/24-30:5) {10};
          
          \node[state,fill=gray!30] () at (17*360/24-15:3) {17};
          \node[state,fill=gray!30] () at (19*360/24-15:3) {19};
          \node[state,fill=gray!30] () at (21*360/24-15:3) {21};
          \node[state,fill=gray!30] () at (18*360/24-30:5) {18};
          \node[state,fill=gray!30] () at (20*360/24-30:5) {20};
          \node[state,fill=gray!30] () at (22*360/24-30:5) {22};
          
          \foreach \Point in {(-2.762, 4.657), (2.692, 4.352), (-1.371, 2.638), (1.499, 2.588), (-0.084, 5.104), (-2.125, 4.437), (2.241, 4.088), (-2.403, 4.244), (4.503, 2.569), (-2.823, 1.314), (1.458, 2.446), (-0.197, 5.006), (-1.478, 2.505), (-2.604, 1.607), (-1.922, 2.528), (0.056, 4.900), (-2.494, 1.485), (-1.824, 2.748), (0.091, 3.352), (1.358, 2.693), (2.731, 1.805), (-2.350, 1.833), (-1.421, 2.458), (-3.162, 1.787), (-4.626, 2.293), (-2.517, 4.110), (-0.138, 4.903), (-2.611, 1.341), (2.711, 4.743), (0.070, 4.873), (2.493, 1.326), (2.560, 4.548), (-2.558, 1.572), (-4.235, 2.486), (-1.607, 2.762), (2.797, 1.405), (-0.208, 2.772), (2.642, 4.266), (-4.226, 2.704), (-4.324, 2.074), (4.033, 2.745), (-1.414, 2.764), (2.603, 4.321), (4.454, 2.537), (-0.045, 4.671), (-0.230, 5.316), (-2.519, 3.905), (4.192, 2.378), (0.024, 4.849), (-1.094, 2.735), (-2.369, 4.256), (0.500, 5.055), (-2.416, 4.101), (2.574, 1.307), (1.319, 2.173), (2.213, 1.448), (0.093, 3.270), (-0.234, 3.413), (4.549, 2.583), (4.058, 2.216), (-4.656, 2.573), (-4.367, 2.550), (1.644, 2.589), (0.065, 2.640), (-2.526, 4.506), (-2.726, 4.403), (-2.777, 4.296), (2.156, 4.612), (-2.689, 1.448), (2.364, 1.535), (2.510, 1.503), (-1.442, 2.572), (2.969, 1.336), (-3.997, 2.414), (2.982, 1.545), (-2.513, 1.100), (0.111, 4.938), (0.054, 2.986), (1.425, 2.652), (-0.246, 3.239), (2.878, 1.468), (-2.238, 4.291)}
            \draw[red,fill=red] \Point circle (0.3ex);
            
          \foreach \Point in {(2.734, -4.289), (2.529, -4.306), (-2.623, -0.981), (-1.280, -2.270), (-3.994, -2.675), (4.413, -2.462), (2.536, -4.509), (-1.347, -2.680), (-2.547, -1.460), (2.086, -2.627), (-2.354, -4.317), (-4.419, -2.148), (1.454, -2.462), (4.332, -2.506), (-4.049, -2.276), (0.407, -3.123), (4.371, -2.586), (2.673, -4.185), (2.737, -1.466), (-0.317, -3.310), (-2.651, -2.036), (-0.098, -5.010), (-0.206, -3.066), (-4.521, -2.550), (0.160, -2.990), (1.592, -2.614), (-1.536, -2.763), (2.587, -4.401), (2.617, -4.297), (0.342, -3.354), (-4.605, -2.686), (4.125, -2.495), (-2.843, -1.521), (-1.193, -2.774), (2.804, -4.269), (0.108, -4.820), (-2.635, -4.758), (1.428, -2.525), (-0.113, -5.325), (-0.056, -2.891), (2.533, -4.701), (2.596, -1.638), (-1.249, -2.872), (-2.429, -1.329), (2.456, -4.382), (0.409, -3.019), (4.552, -2.567), (4.099, -2.143), (-1.361, -2.789), (0.080, -2.593), (-2.508, -4.493), (4.649, -2.833), (-3.121, -1.514), (1.504, -2.515), (1.568, -2.829), (-2.384, -4.472), (1.406, -2.315), (-2.719, -4.493), (0.074, -4.835), (1.671, -2.135), (0.243, -4.852), (-0.012, -3.173), (-4.507, -2.653), (-2.554, -4.155), (-0.194, -4.828), (1.991, -2.661), (2.774, -4.225), (-2.511, -1.582), (1.658, -2.804), (2.478, -4.146), (0.057, -3.024), (2.208, -1.264), (-4.032, -2.924), (-2.242, -4.243), (-2.459, -1.367), (0.210, -5.074), (-2.486, -1.536), (4.116, -2.506), (1.160, -2.663), (2.498, -1.256)}
            \draw[green,fill=green] \Point circle (0.3ex);
            
          \foreach \Point in {(-5.263, -0.033), (5.043, -0.004), (-2.903, 0.157), (-2.792, -0.179), (-3.158, -0.039), (-4.781, -0.270), (-2.846, 0.263), (3.056, 0.176), (5.137, 0.204), (3.117, 0.171), (4.656, 0.193), (4.821, 0.141), (4.947, -0.164), (4.744, -0.308), (2.670, 0.021), (5.164, -0.035), (-4.919, 0.106), (-3.103, -0.251), (5.343, -0.055), (2.738, 0.186), (-5.471, 0.019), (5.113, 0.100), (-4.743, 0.308), (-3.055, -0.181), (3.523, -0.277), (3.263, -0.044), (-4.787, 0.029), (-5.120, 0.017), (4.969, 0.051), (-4.901, -0.142), (5.159, 0.148), (-4.921, -0.125), (-5.186, -0.329), (-3.113, -0.403), (5.095, 0.242), (-3.330, 0.139), (-3.365, -0.221), (4.836, -0.208)}
            \draw[yellow,fill=yellow] \Point circle (0.3ex);
        \end{tikzpicture}}
    \end{minipage}
    \hspace*{0.5ex}
    \begin{minipage}[t]{0.23\textwidth}
        \centering
        \subfiguretitle{(b)}
        \vspace*{0.8ex}
        \resizebox{\textwidth}{!}{%
        \begin{tikzpicture}[
                        >= stealth, 
                        semithick 
                    ]
                    \tikzstyle{every state}=[
                        draw = black,
                        thick,
                        fill = white,
                        inner sep=0pt,
                        text width=6mm,
                        align=center,
                        scale=0.6
                    ]
          \foreach \x in {1,3,...,24}
            \node[state] (\x) at (\x*360/24-15:3) {\x};
          \foreach \x in {2,4,...,24}
            \node[state] (\x) at (\x*360/24-30:5) {\x};
          \foreach \x/\y in {1/2,3/4,5/6,7/8,9/10,11/12,13/14,15/16,17/18,19/20,21/22,23/24, 8/10,10/12,7/9,9/11,20/22,22/24,19/21,21/23}
            \draw [-] (\x) -- (\y);
          \foreach \x/\y in {4/6,6/8,3/5,5/7,4/2,2/24,3/1,1/23,16/14,14/12,15/13,13/11,16/18,18/20,15/17,17/19}
            \draw [-{Stealth[length=3mm, width=2mm]}] (\x) -- (\y);
          
          \node[state,fill=gray!30] () at (7*360/24-15:3) {7};
          \node[state,fill=gray!30] () at (9*360/24-15:3) {9};
          \node[state,fill=gray!30] () at (11*360/24-15:3) {11};
          \node[state,fill=gray!30] () at (8*360/24-30:5) {8};
          \node[state,fill=gray!30] () at (10*360/24-30:5) {10};
          \node[state,fill=gray!30] () at (12*360/24-30:5) {12};
          
          \node[state,fill=gray!30] () at (19*360/24-15:3) {19};
          \node[state,fill=gray!30] () at (21*360/24-15:3) {21};
          \node[state,fill=gray!20] () at (23*360/24-15:3) {23};
          \node[state,fill=gray!20] () at (20*360/24-30:5) {20};
          \node[state,fill=gray!20] () at (22*360/24-30:5) {22};
          \node[state,fill=gray!20] () at (24*360/24-30:5) {24};
          
          \foreach \Point in {(-0.031, 3.032), (0.307, 2.985), (-4.208, 2.428), (-2.569, 4.652), (-2.445, 4.017), (0.385, 4.821), (0.070, 5.168), (0.029, 3.073), (-0.279, 5.259), (0.058, 4.734), (0.123, 5.160), (0.241, 3.092), (-2.582, 1.286), (-2.584, 1.262), (-4.817, 2.471), (-2.471, 4.628), (-2.639, 4.015), (-0.321, 5.236), (-2.507, 1.563), (-1.627, 2.543), (-0.146, 5.038), (-1.096, 2.423), (-2.570, 4.352), (-1.636, 2.811), (-2.493, 1.623), (0.048, 4.841), (-2.622, 1.766), (-1.611, 2.683), (-2.396, 4.421), (-2.365, 1.652), (-2.514, 4.232), (-2.570, 1.426), (-0.589, 2.888), (0.198, 4.922), (-1.511, 2.518), (-2.265, 3.864), (-1.390, 2.003), (-2.125, 4.231), (-2.614, 1.722), (-4.734, 2.367), (-2.798, 1.761), (-0.001, 5.127), (-2.355, 4.334), (0.228, 2.943), (-0.358, 5.109), (-3.978, 2.098), (-0.296, 4.863), (-2.586, 4.153), (0.127, 2.876), (-0.083, 4.969), (0.112, 2.802), (-2.278, 4.362), (0.107, 3.025), (-4.334, 2.579), (-4.408, 2.635), (-2.143, 4.382), (0.202, 4.921), (0.038, 4.993), (-0.108, 5.077), (0.092, 5.141), (0.023, 2.665), (-4.120, 2.781), (-2.795, 4.367), (-1.370, 2.504), (-1.523, 2.808), (-2.561, 1.554), (-1.297, 2.645), (-2.587, 1.723), (-2.558, 4.467), (0.196, 2.870), (-1.729, 2.724), (-2.468, 4.589), (-4.391, 2.513), (-0.211, 4.781), (-2.341, 4.574), (-4.471, 2.225), (-2.668, 1.035), (-2.475, 4.201), (-2.895, 4.677), (0.328, 3.426), (-2.629, 1.271), (-0.083, 5.079)}
            \draw[red,fill=red] \Point circle (0.3ex);
            
          \foreach \Point in {(1.803, -4.808), (2.491, -4.173), (0.078, -4.904), (2.677, -4.288), (3.905, -2.574), (1.360, -2.319), (2.326, -1.589), (0.056, -2.938), (2.266, -4.605), (4.362, -2.661), (0.224, -2.658), (4.448, -2.217), (0.020, -3.234), (2.804, -4.599), (2.672, -4.619), (2.509, -4.387), (2.421, -1.582), (2.286, -4.331), (0.213, -2.871), (1.772, -2.564), (-0.197, -5.335), (0.057, -5.101), (1.509, -2.506), (2.865, -4.242), (-0.336, -5.237), (4.421, -2.218), (1.769, -2.389), (4.100, -2.706), (0.072, -3.226), (4.632, -2.425), (0.164, -3.015), (2.806, -4.105), (1.692, -2.636), (0.203, -2.649), (1.296, -2.620), (2.498, -4.601), (4.150, -2.325), (1.524, -2.683), (4.334, -2.451), (2.644, -1.370), (2.592, -1.551), (2.504, -1.466), (4.588, -2.509), (0.046, -3.032), (0.103, -5.010), (2.915, -1.480), (4.246, -2.546), (2.506, -1.340), (4.356, -2.408), (2.766, -4.259), (0.231, -5.364), (-0.132, -3.042), (4.254, -2.429), (0.057, -2.898), (1.392, -2.594), (-1.717, -2.648), (3.003, -1.537), (2.700, -1.619), (2.463, -1.657), (2.272, -1.657), (2.415, -3.826), (0.129, -3.196), (2.386, -4.085), (1.356, -2.609), (1.543, -2.558), (-0.119, -4.996), (0.066, -4.875), (1.281, -2.609), (1.350, -2.585), (1.256, -2.711), (4.141, -2.586), (0.381, -2.761), (2.358, -4.322), (1.409, -2.557), (2.559, -3.954), (-0.278, -5.137), (-0.028, -4.908), (0.616, -2.756), (1.235, -2.512), (4.123, -2.697)}
            \draw[green,fill=green] \Point circle (0.3ex);
            
          \foreach \Point in {(-2.397, 1.408), (-0.138, 4.988), (-0.363, -2.748), (-0.111, 3.148), (3.024, -1.229), (-0.241, -5.133), (-2.651, 1.268), (2.839, -4.357), (-1.538, 2.697), (-2.753, 1.566), (-1.434, 2.096), (-0.281, 2.462), (4.156, -2.461), (0.032, -4.903), (1.599, -2.834), (1.118, -2.610), (2.767, -4.522), (-2.287, 1.407), (4.394, -2.621), (-2.915, 1.448), (0.107, -4.814), (2.200, -4.105), (2.331, -4.324), (-2.010, 4.506), (-2.529, 4.464), (0.119, -5.037), (0.278, 5.139), (-0.429, -2.931), (0.413, 5.365), (-1.473, 2.842), (-1.955, 2.701), (2.190, -1.405), (-2.514, 4.125), (2.543, -4.680), (0.132, -3.002), (0.097, 3.526), (-0.331, 5.359), (2.329, -4.049)}
            \draw[yellow,fill=yellow] \Point circle (0.3ex);
        \end{tikzpicture}}
    \end{minipage}
    \hspace*{0.5ex}
    \begin{minipage}[t]{0.238\textwidth}
        \centering
        \subfiguretitle{(c)}
        \vspace*{1ex}
        \resizebox{\textwidth}{!}{%
        \begin{tikzpicture}[
                        >= stealth, 
                        semithick 
                    ]
                    \tikzstyle{every state}=[
                        draw = black,
                        thick,
                        fill = white,
                        inner sep=0pt,
                        text width=6mm,
                        align=center,
                        scale=0.6
                    ]
          \foreach \x in {1,3,...,24}
            \node[state] (\x) at (\x*360/24-15:3) {\x};
          \foreach \x in {2,4,...,24}
            \node[state] (\x) at (\x*360/24-30:5) {\x};
          \foreach \x/\y in {1/2,3/4,5/6,7/8,9/10,11/12,13/14,15/16,17/18,19/20,21/22,23/24, 10/12,12/14,9/11,11/13,22/24,24/2,21/23,23/1}
            \draw [-] (\x) -- (\y);
          \foreach \x/\y in {6/8,8/10,5/7,7/9,6/4,4/2,5/3,3/1,18/16,16/14,17/15,15/13,18/20,20/22,17/19,19/21}
            \draw [-{Stealth[length=3mm, width=2mm]}] (\x) -- (\y);
          
          \node[state,fill=gray!30] () at (9*360/24-15:3) {9};
          \node[state,fill=gray!30] () at (11*360/24-15:3) {11};
          \node[state,fill=gray!30] () at (13*360/24-15:3) {13};
          \node[state,fill=gray!30] () at (10*360/24-30:5) {10};
          \node[state,fill=gray!30] () at (12*360/24-30:5) {12};
          \node[state,fill=gray!30] () at (14*360/24-30:5) {14};
          
          \node[state,fill=gray!20] () at (21*360/24-15:3) {21};
          \node[state,fill=gray!30] () at (23*360/24-15:3) {23};
          \node[state,fill=gray!30] () at (1*360/24-15:3) {1};
          \node[state,fill=gray!30] () at (22*360/24-30:5) {22};
          \node[state,fill=gray!30] () at (24*360/24-30:5) {24};
          \node[state,fill=gray!30] () at (2*360/24-30:5) {2};

          \foreach \Point in {(-3.144, 0.296), (-1.266, 2.698), (-4.092, 2.691), (-2.986, -0.076), (-3.166, 1.387), (-1.246, 2.358), (-4.346, 2.159), (-3.252, 0.449), (-2.318, 4.241), (-2.841, 0.083), (-2.541, 1.996), (-2.670, 1.506), (-2.373, 1.581), (-2.811, 0.025), (-2.939, 1.387), (-2.413, 1.500), (-1.136, 2.904), (-2.306, 4.224), (-4.957, 0.060), (-2.544, 4.699), (-4.530, 2.337), (-4.699, -0.245), (-1.315, 2.721), (-2.432, 1.660), (-4.424, 2.839), (-4.861, 0.030), (-2.984, 0.070), (-2.541, 4.253), (-2.972, 1.612), (-2.346, 4.153), (-4.891, -0.338), (-3.252, 0.009), (-4.383, 2.624), (-2.573, 4.634), (-1.896, 2.742), (-2.789, 1.491), (-2.500, 1.723), (-3.932, 2.773), (-5.188, 0.260), (-4.882, 0.068), (-1.385, 2.776), (-4.963, 0.145), (-2.623, 1.270), (-3.004, 0.007), (-1.358, 2.678), (-1.707, 3.066), (-2.397, 4.012), (-4.155, 2.235), (-5.478, 0.205), (-4.305, 2.480), (-3.894, 2.576), (-4.867, -0.142), (-2.623, 1.192), (-2.887, 1.400), (-2.688, 1.471), (-2.406, 1.255), (-2.387, 4.604), (-3.942, 2.284), (-1.693, 2.202), (-1.588, 2.544), (-2.621, 1.240), (-2.740, 1.570), (-3.112, 0.220), (-2.615, 0.147), (-2.365, 1.343), (-1.585, 2.807), (-2.565, 4.388), (-2.638, 1.622), (-2.833, 0.092), (-1.797, 2.931), (-2.741, 1.643), (-2.756, 1.422), (-1.826, 2.937), (0.165, 2.916), (-3.127, -0.078), (-2.633, -0.192), (-4.891, -0.059), (-2.454, 4.555), (-2.857, 0.143), (-2.368, 1.221), (-2.762, 1.431), (-2.441, 4.389)}
            \draw[red,fill=red] \Point circle (0.3ex);
            
          \foreach \Point in {(5.004, -0.151), (4.770, -2.620), (2.572, -4.352), (2.989, 0.169), (4.254, -2.661), (2.261, -4.412), (2.332, -4.075), (2.704, -1.818), (4.360, -2.398), (2.354, -4.533), (2.666, -1.555), (1.486, -2.605), (2.653, -1.516), (2.841, -1.349), (2.630, -1.600), (4.629, -2.551), (1.724, -2.527), (1.515, -2.849), (1.322, -2.543), (1.372, -2.556), (2.133, -4.401), (4.468, -2.593), (2.646, -0.042), (2.555, -1.249), (0.426, -2.918), (1.707, -2.581), (1.378, -2.729), (2.466, -4.223), (4.280, -2.297), (4.888, -0.050), (2.646, -4.316), (2.380, -1.469), (2.651, -4.500), (3.043, 0.166), (2.989, 0.093), (2.684, -1.572), (2.495, -1.524), (5.294, -0.068), (2.655, -4.062), (2.485, -4.247), (2.860, -1.378), (1.756, -2.454), (3.255, -0.468), (2.641, -1.590), (2.803, -1.782), (4.080, -2.631), (4.785, -0.125), (2.326, -1.299), (2.640, -4.709), (4.945, -0.152), (4.228, -2.199), (4.218, -2.580), (2.962, 0.118), (1.582, -2.649), (2.963, -1.610), (3.184, 0.038), (2.910, -0.046), (2.112, -1.388), (2.750, 0.141), (3.055, -0.000), (2.684, -1.325), (5.042, 0.232), (4.906, 0.157), (2.275, -4.611), (5.016, -0.175), (3.116, -0.042), (2.760, -4.243), (1.825, -2.605), (1.435, -2.702), (2.628, -4.353), (1.291, -2.728), (4.674, -2.228), (2.427, -4.336), (2.654, -0.039), (2.529, -1.686), (3.101, -0.100), (4.821, -2.152), (2.734, -1.411), (2.724, -1.507), (3.175, -0.240)}
            \draw[green,fill=green] \Point circle (0.3ex);
            
          \foreach \Point in {(-3.204, 0.039), (-4.423, 2.509), (4.287, -2.492), (-1.430, 2.393), (4.347, -2.728), (4.347, -2.511), (-2.616, 1.455), (2.671, -4.547), (-1.509, 2.705), (-4.432, 2.497), (-4.751, -0.183), (-1.640, 2.586), (2.329, -4.312), (5.338, 0.061), (3.131, 0.222), (2.653, -1.883), (2.209, -1.011), (-3.133, -0.072), (2.388, -4.363), (-2.459, 4.547), (2.610, -1.403), (5.204, -0.422), (1.024, -2.806), (-2.801, 0.034), (-4.798, 0.044), (2.670, -1.390), (-2.451, 1.464), (2.242, -1.549), (-3.896, 2.387), (-4.576, 2.639), (-2.613, 1.695), (4.097, -2.850), (-5.267, 0.192), (1.807, -2.469), (4.415, -2.422), (-1.475, 2.578), (-2.966, -0.002), (4.648, -2.678)}
            \draw[yellow,fill=yellow] \Point circle (0.3ex);
        \end{tikzpicture}}
    \end{minipage}
    \hspace*{0.5ex}
    \begin{minipage}[t]{0.23\textwidth}
        \centering
        \subfiguretitle{(d)}
        \vspace*{1ex}
        \resizebox{\textwidth}{!}{%
        \begin{tikzpicture}[
                        >= stealth, 
                        semithick 
                    ]
                    \tikzstyle{every state}=[
                        draw = black,
                        thick,
                        fill = white,
                        inner sep=0pt,
                        text width=6mm,
                        align=center,
                        scale=0.6
                    ]
          \node[state,fill=yellow!60] (1) at (1*360/24-15:3) {1};
          \node[state,fill=yellow!60] (2) at (2*360/24-30:5) {2};
          \node[state,fill=red!60] (3) at (3*360/24-15:3) {3};
          \node[state,fill=red!60] (4) at (4*360/24-30:5) {4};
          \node[state,fill=red!60] (5) at (5*360/24-15:3) {5};
          \node[state,fill=red!60] (6) at (6*360/24-30:5) {6};
          \node[state,fill=red!60] (7) at (7*360/24-15:3) {7};
          \node[state,fill=red!60] (8) at (8*360/24-30:5) {8};
          \node[state,fill=red!60] (9) at (9*360/24-15:3) {9};
          \node[state,fill=red!60] (10) at (10*360/24-30:5) {10};
          \node[state,fill=red!60] (11) at (11*360/24-15:3) {11};
          \node[state,fill=red!60] (12) at (12*360/24-30:5) {12};
          \node[state,fill=yellow!60] (13) at (13*360/24-15:3) {13};
          \node[state,fill=yellow!60] (14) at (14*360/24-30:5) {14};
          \node[state,fill=green!60] (15) at (15*360/24-15:3) {15};
          \node[state,fill=green!60] (16) at (16*360/24-30:5) {16};
          \node[state,fill=green!60] (17) at (17*360/24-15:3) {17};
          \node[state,fill=green!60] (18) at (18*360/24-30:5) {18};
          \node[state,fill=green!60] (19) at (19*360/24-15:3) {19};
          \node[state,fill=green!60] (20) at (20*360/24-30:5) {20};
          \node[state,fill=green!60] (21) at (21*360/24-15:3) {21};
          \node[state,fill=green!60] (22) at (22*360/24-30:5) {22};
          \node[state,fill=green!60] (23) at (23*360/24-15:3) {23};
          \node[state,fill=green!60] (24) at (24*360/24-30:5) {24};
          
          \foreach \x/\y in {1/2,3/4,5/6,7/8,9/10,11/12,13/14,15/16,17/18,19/20,21/22,23/24, 6/8,8/10,5/7,7/9,18/20,20/22,17/19,19/21}
            \draw [-] (\x) -- (\y);
          \foreach \x/\y in {2/4,4/6,1/3,3/5,2/24,24/22,1/23,23/21,14/12,12/10,13/11,11/9,14/16,16/18,13/15,15/17}
            \draw [-{Stealth[length=3mm, width=2mm]}] (\x) -- (\y);
        \end{tikzpicture}}
    \end{minipage} \\[1ex]
    \begin{minipage}[t]{0.294\textwidth}
        \centering
        \subfiguretitle{(e)}
        \vspace*{1ex}
        \includegraphics[width=\linewidth]{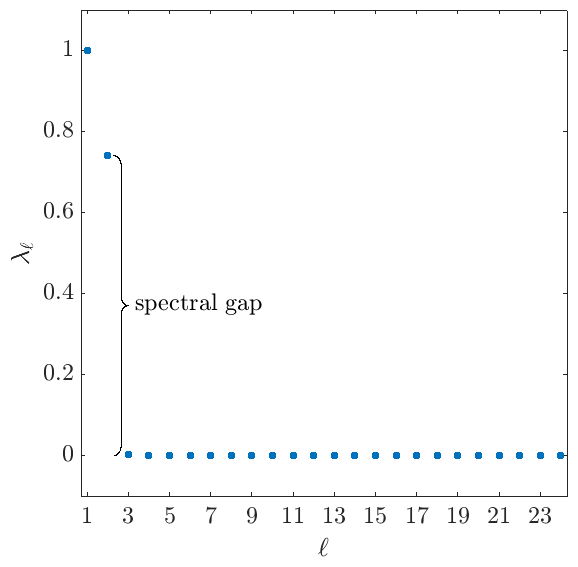}
    \end{minipage}
    \begin{minipage}[t]{0.39\textwidth}
        \centering
        \subfiguretitle{(f)}
        \vspace*{1ex}
        \includegraphics[width=\linewidth]{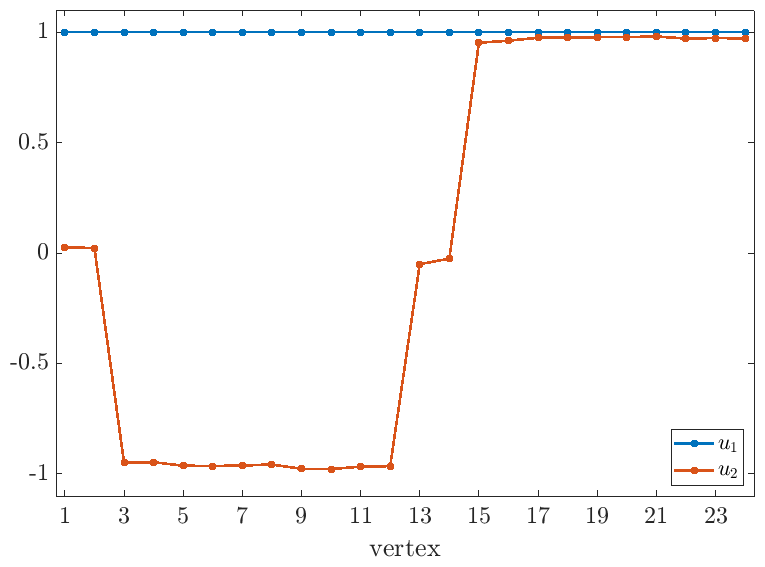}
    \end{minipage}
    \begin{minipage}[t]{0.291\textwidth}
        \centering
        \subfiguretitle{(g)}
        \vspace*{0.55ex}
        \includegraphics[width=\linewidth]{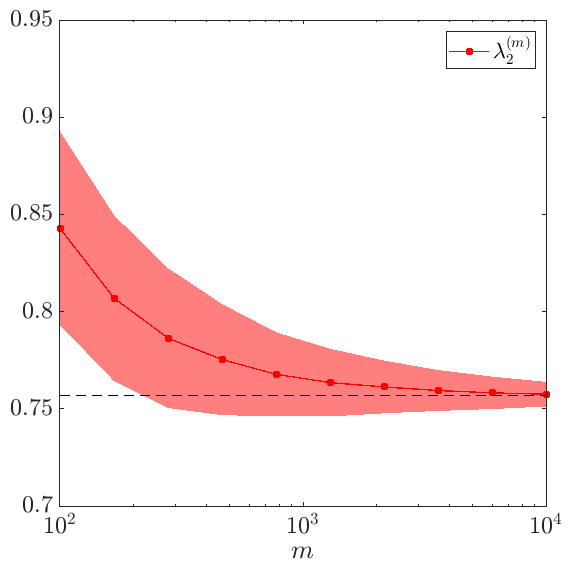}
    \end{minipage}
    \caption{Time-evolving graph with two attracting sets (marked in light gray) at times (a)~$ t = 0 $, (b) $ t = 10 $, and (c) $ t = 20 $. All edge weights are one. Note that some edges are directed and some undirected. After every ten steps, the edges of the graph are `rotated' counterclockwise. The red, green, and yellow dots represent---initially uniformly distributed---random walkers colored according to the coherent sets computed below. (For visualization purposes, we added Gaussian noise to the positions of the random walkers. Nevertheless, the state space is discrete and the random walkers are always assigned to one of the vertices.)~(d)~Clustering of the time-evolving graph at time $ t = 0 $ using SEBA. Random walkers starting in a coherent set will typically be trapped within the set for a long time. The yellow nodes are not assigned to any coherent set and can be regarded as transition regions. The corresponding random walkers will end up in either the red or the green cluster. (e)~Eigenvalues of the matrix $ Q $ estimated from trajectory data. Two eigenvalues are close to 1, which implies that there are two coherent sets. (f)~Eigenvectors corresponding to the two dominant eigenvalues. (g) Convergence of the second eigenvalue using approach A to the eigenvalue computed with the aid of Approach B. The red shaded area represents the standard deviation.}
    \label{fig:Double-well graph}
\end{figure}

Let us now analyze the time-evolving graph shown in Figure~\ref{fig:Double-well graph}(a)--(c). The graph can be viewed as a discretization of a two-dimensional double-well problem with rotating wells. Random walkers will quickly move to one of the two attracting sets and follow the movement of these time-dependent clusters. The behavior of random walkers starting in the same cluster is hence \emph{coherent}. We estimate the forward-backward transition matrix $ Q $ by computing $ F_\tau^{(m)} $ using $ m = 5000 $ random walks of length 100 and then compute the eigenvalues and eigenvectors. Since the matrix $ C_{yy} $ is singular in this case (there are no random walkers in some vertices at the final time), we have to use Tikhonov regularization (or, alternatively, the pseudoinverse). We choose $ \varepsilon = 10^{-8} $. In order to detect coherent sets, we apply SEBA to the dominant two eigenvectors. The results are shown in Figure~\ref{fig:Double-well graph}(d)--(f). Unlike $ k $-means, SEBA does not assign all vertices to clusters. In addition to the coherent sets (marked in red and green), we obtain a transition region (marked in yellow). Figure~\ref{fig:Double-well graph}(g) illustrates the dependence of the estimated second eigenvalue on the number of random walkers. The convergence rate in this case is approximately $ \mathcal{O}\left(m^{-1}\right) $, although in general we can only expect $ \mathcal{O}\left(m^{-\nicefrac{1}{2}}\right) $ since EDMD and its extensions are based on Monte Carlo approximations, see also \cite{WKR15}. \exampleSymbol
\end{example}

The coherence of the sets detected in Example~\ref{ex:Double-well graph} can be decreased by adding reverse edges with a low edge weight to all directed edges. This allows random walkers to transition between the coherent sets more easily. It is important to stress here that the notion of \emph{coherence} is different from the standard interpretation of clusters: Spectral clustering algorithms for static undirected graphs identify fixed sets of vertices with the property that random walkers will remain in such a set for a long time---with a probability that is related to the dominant eigenvalues of the random-walk Laplacian---before moving to another cluster. Coherent sets, on the other hand, are time-evolving structures. Random walkers trapped in such time-dependent sets will with a high probability---now determined by the dominant eigenvalues of the forward-backward Laplacian---move in a coherent way.

\begin{example} \label{ex:Quadruple-gyre graph}

\begin{figure}
    \centering
    \begin{minipage}[t]{0.25\textwidth}
        \centering
        \subfiguretitle{(a)}
        \includegraphics[width=0.97\linewidth]{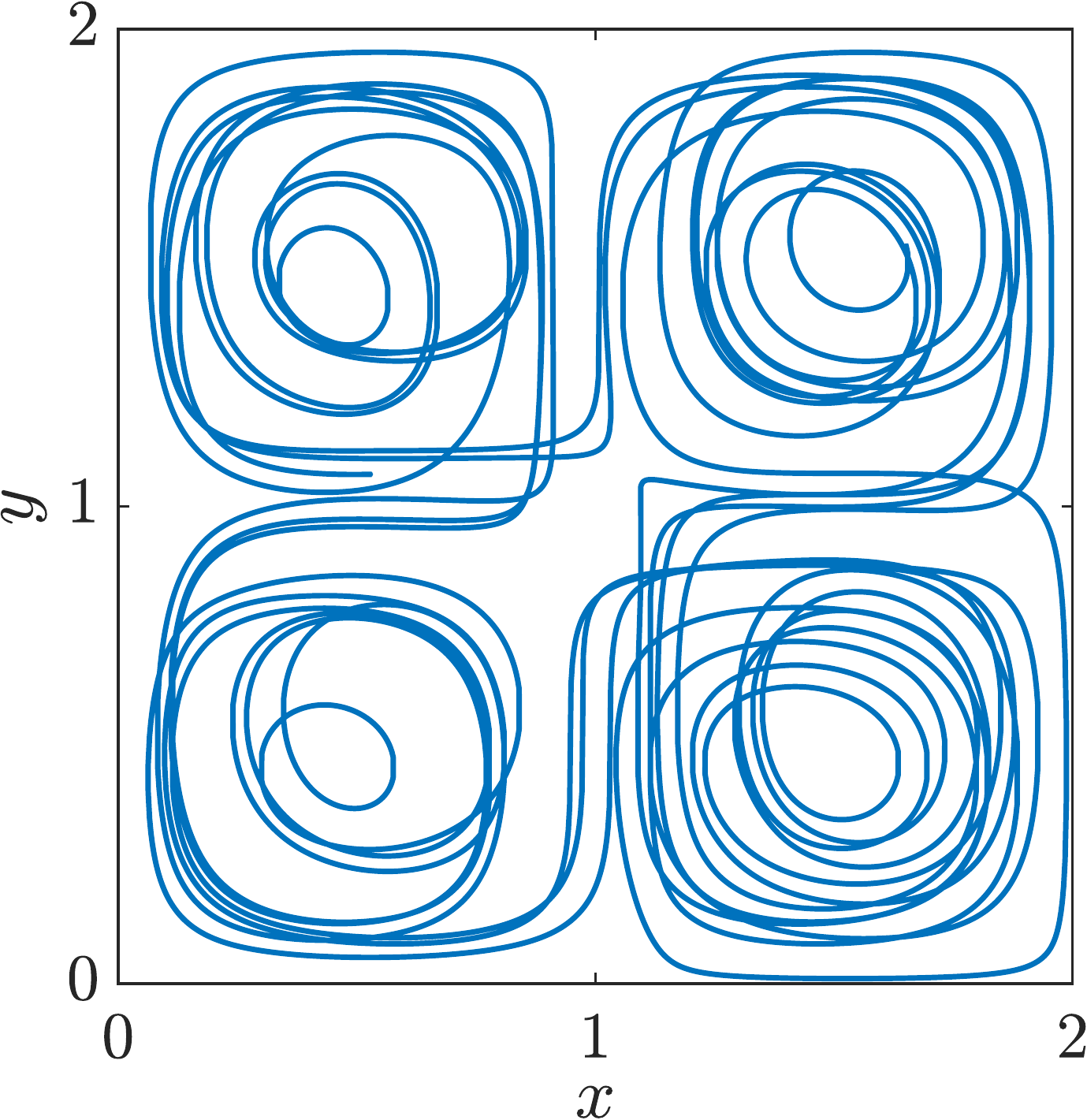}
    \end{minipage}
    \begin{minipage}[t]{0.24\textwidth}
        \centering
        \subfiguretitle{(b) $ t = 0 $}
        \vspace*{0.5ex}
        \includegraphics[width=0.9\linewidth]{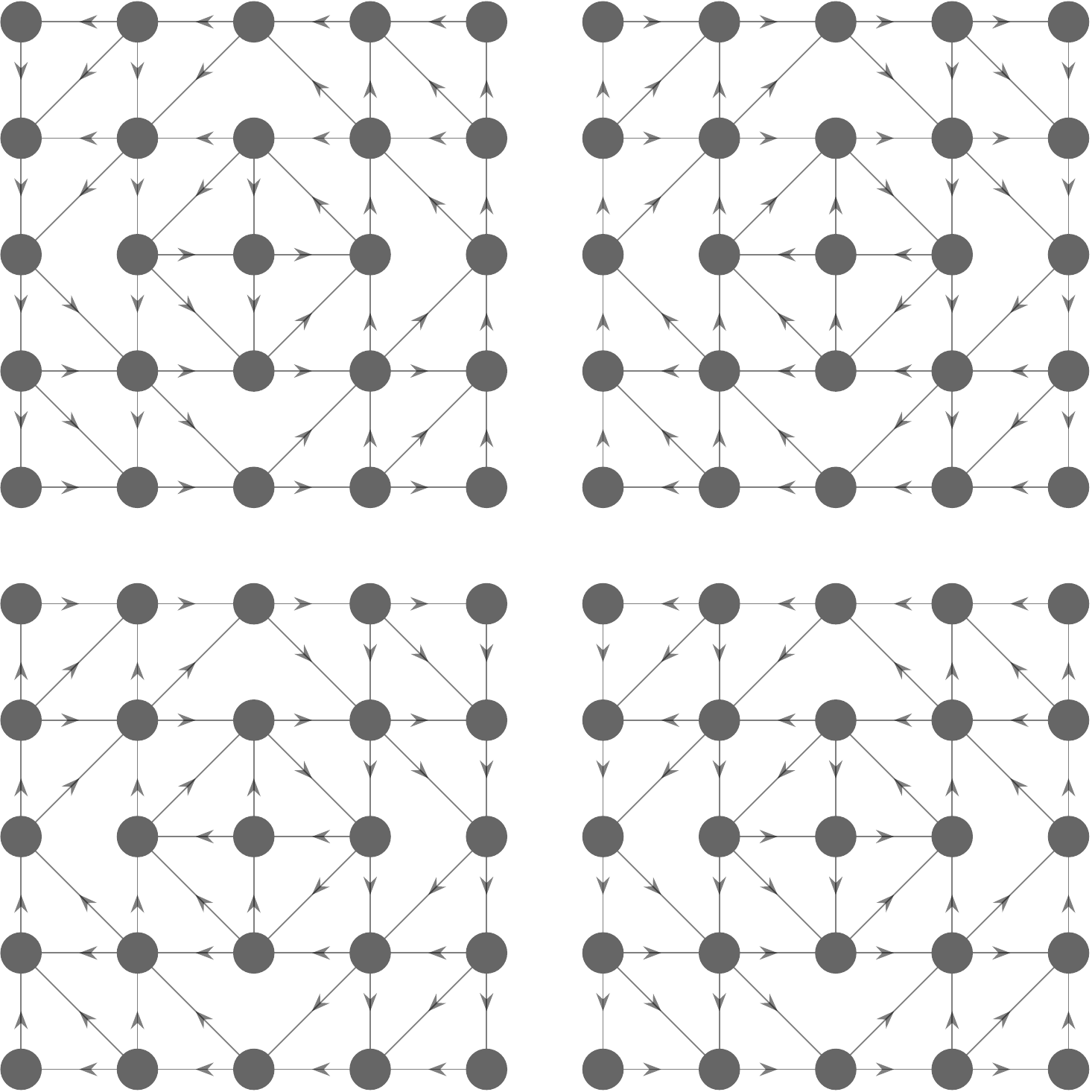}
    \end{minipage}
    \begin{minipage}[t]{0.24\textwidth}
        \centering
        \subfiguretitle{(c) $ t = 1 $}
        \vspace*{0.5ex}
        \includegraphics[width=0.9\linewidth]{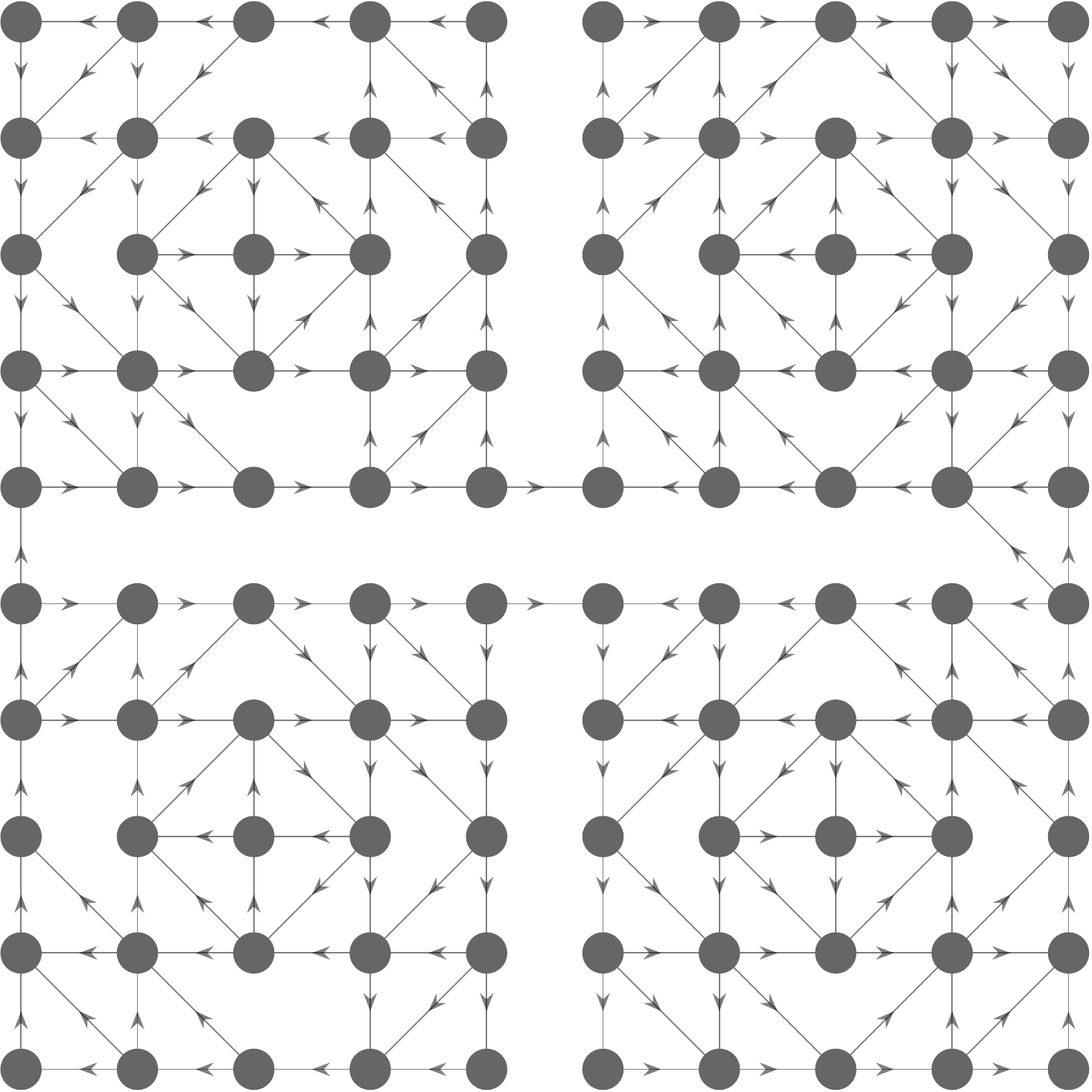}
    \end{minipage}
    \begin{minipage}[t]{0.24\textwidth}
        \centering
        \subfiguretitle{(d) $ t = 4 $}
        \vspace*{0.5ex}
        \includegraphics[width=0.9\linewidth]{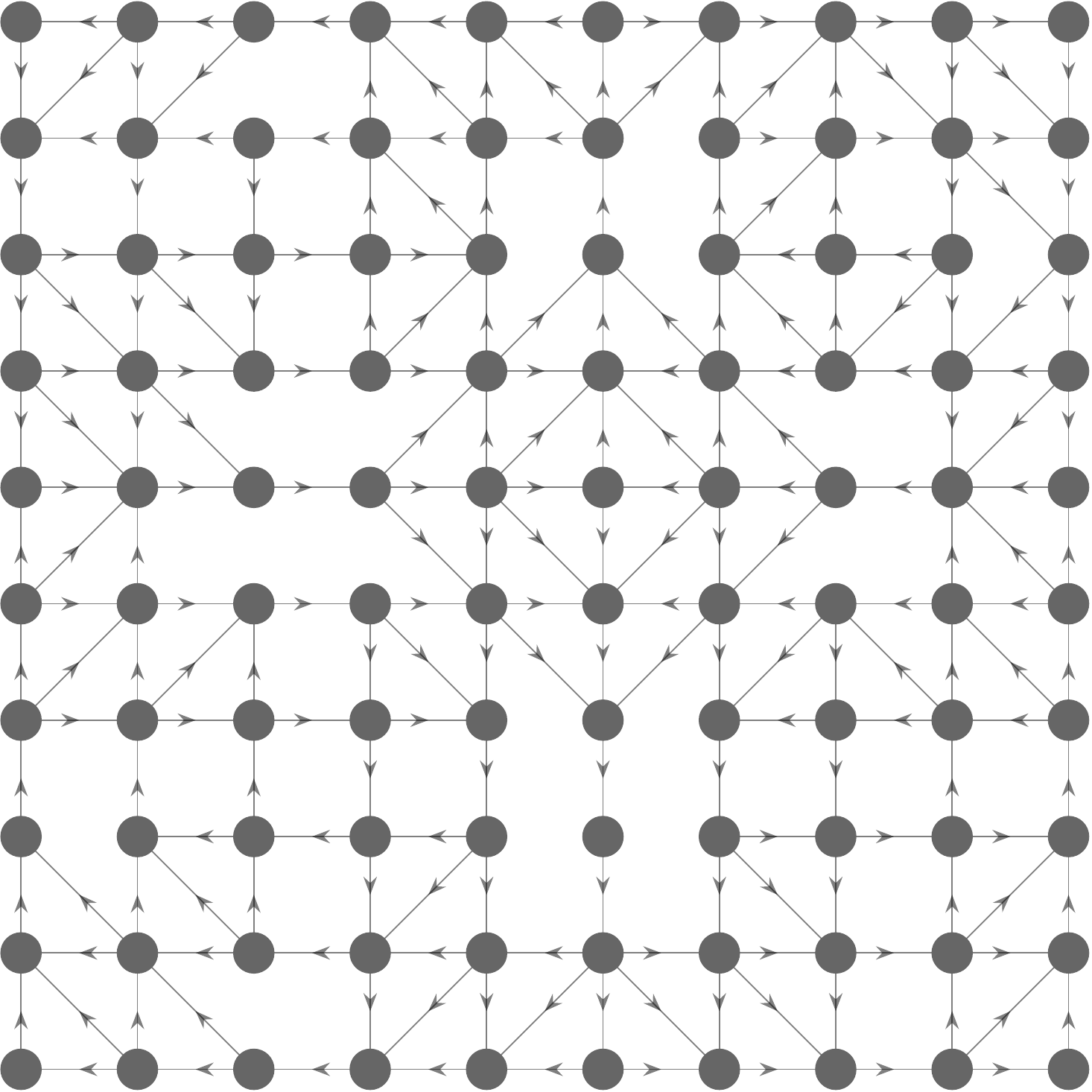}
    \end{minipage} \\[1ex]
    \begin{minipage}[t]{0.25\textwidth}
        \centering
        \subfiguretitle{(h)}
        \vspace*{0.7ex}
        \includegraphics[width=0.99\linewidth]{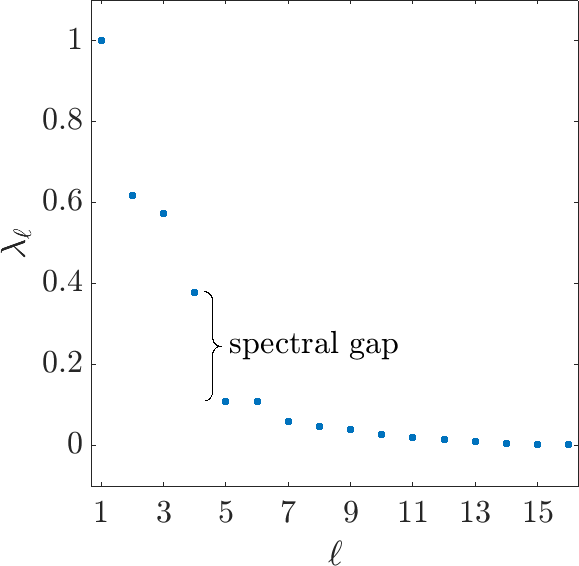}
    \end{minipage}
    \begin{minipage}[t]{0.24\textwidth}
        \centering
        \subfiguretitle{(e) $ t = 10 $}
        \vspace*{0.5ex}
        \includegraphics[width=0.9\linewidth]{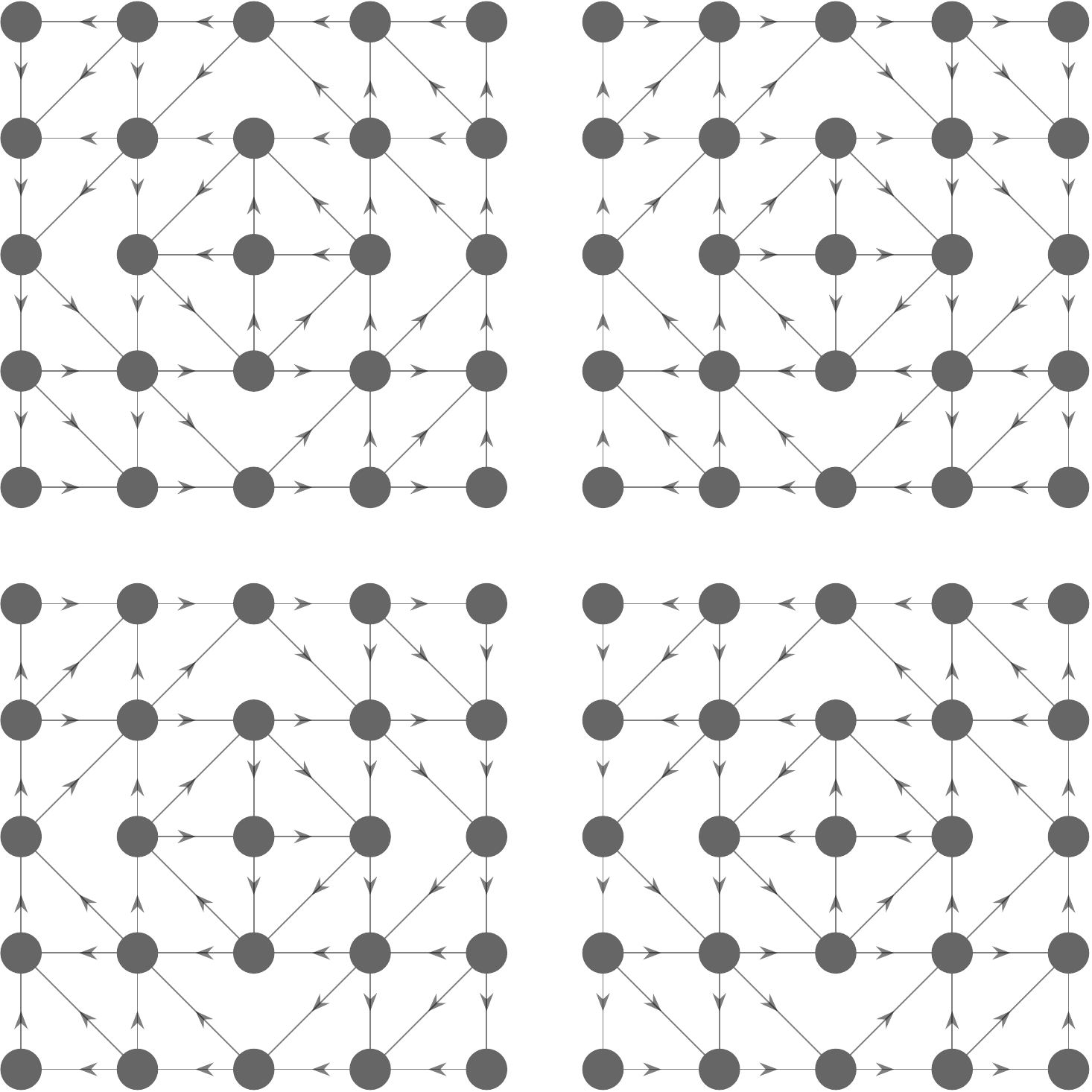}
    \end{minipage}
    \begin{minipage}[t]{0.24\textwidth}
        \centering
        \subfiguretitle{(f) $ t = 11 $}
        \vspace*{0.5ex}
        \includegraphics[width=0.9\linewidth]{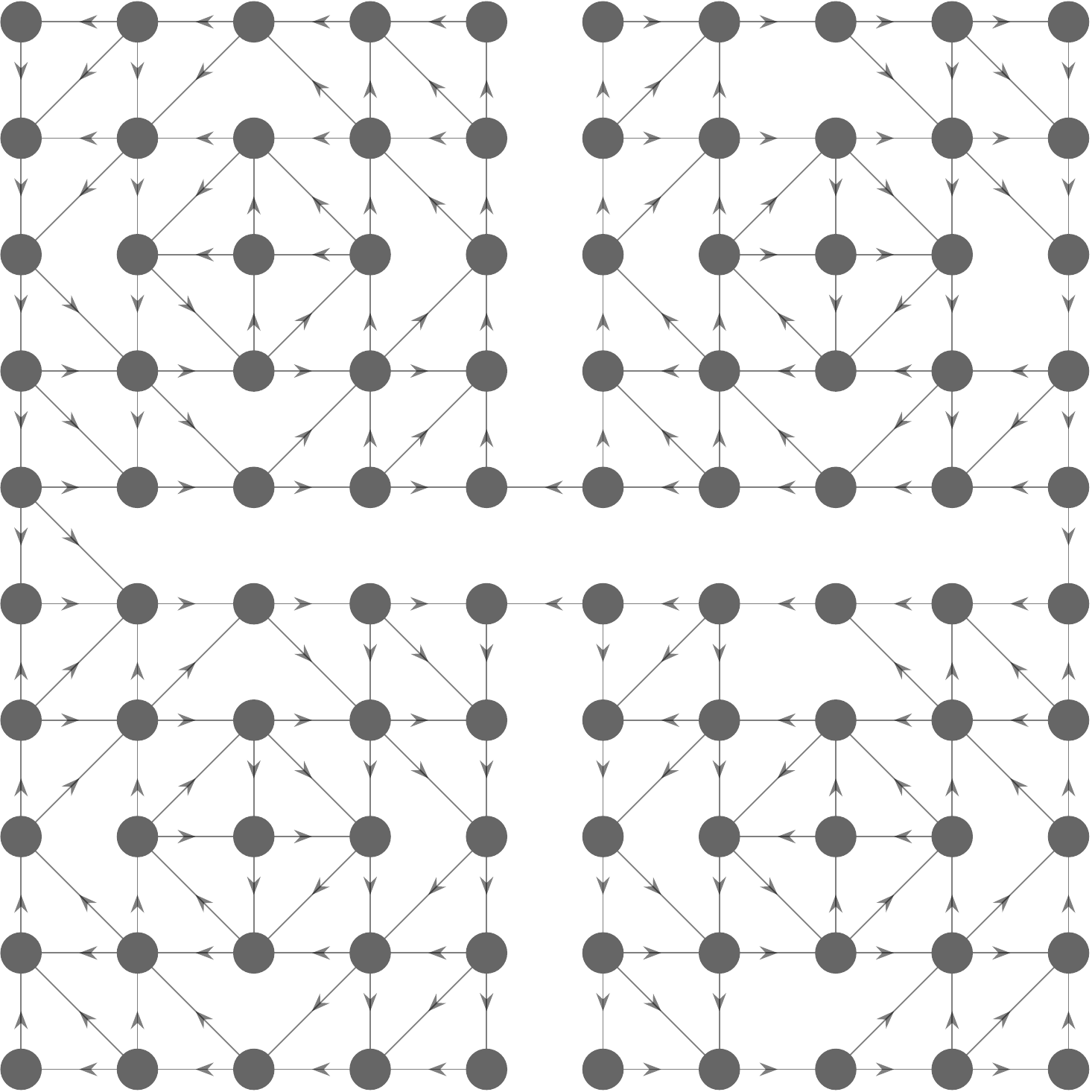}
    \end{minipage}
    \begin{minipage}[t]{0.24\textwidth}
        \centering
        \subfiguretitle{(g) $ t = 14 $}
        \vspace*{0.5ex}
        \includegraphics[width=0.9\linewidth]{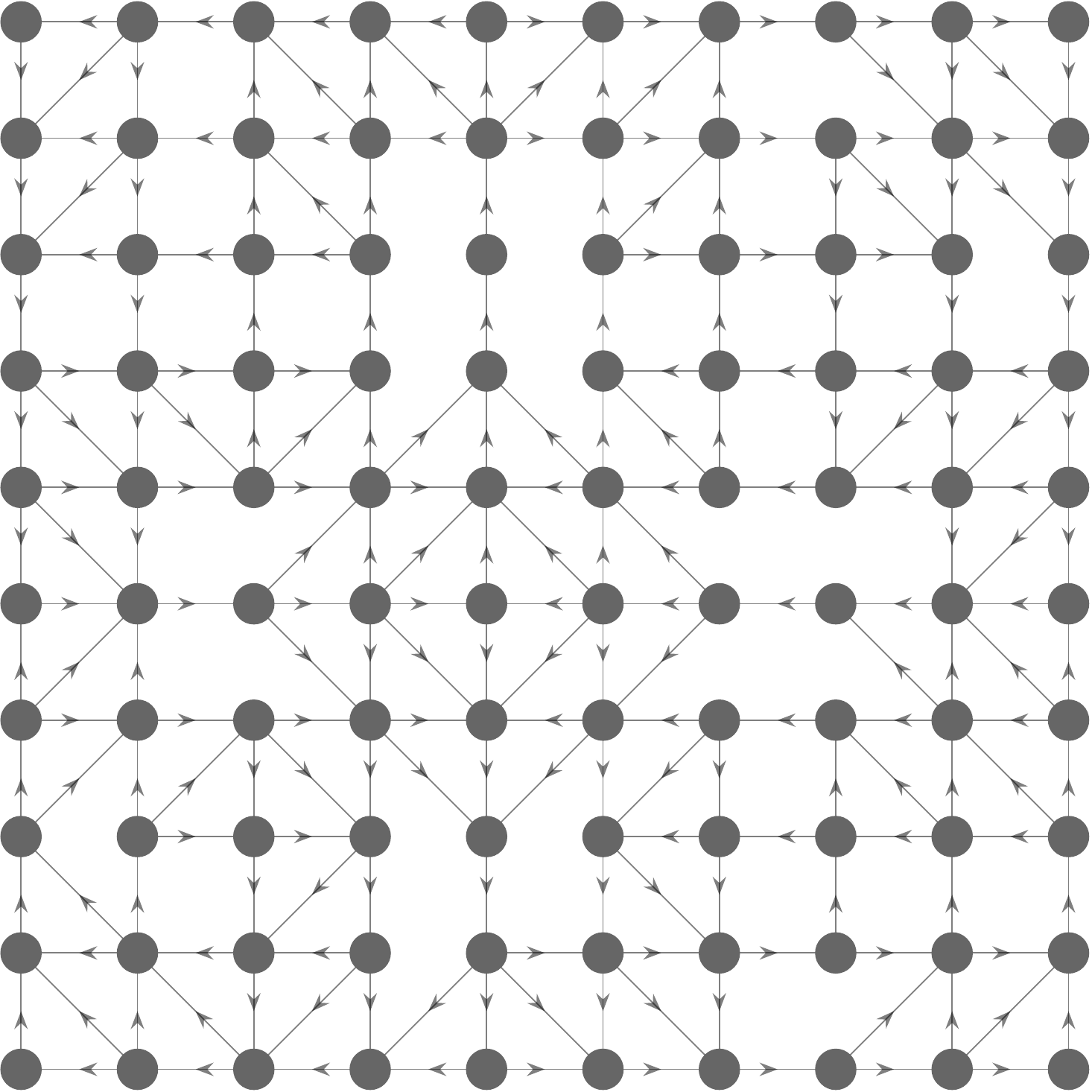}
    \end{minipage} \\[1ex]
    \begin{minipage}[t]{0.49\textwidth}
        \centering
        \subfiguretitle{(i)}
        \vspace*{0.5ex}
        \includegraphics[width=0.7\linewidth]{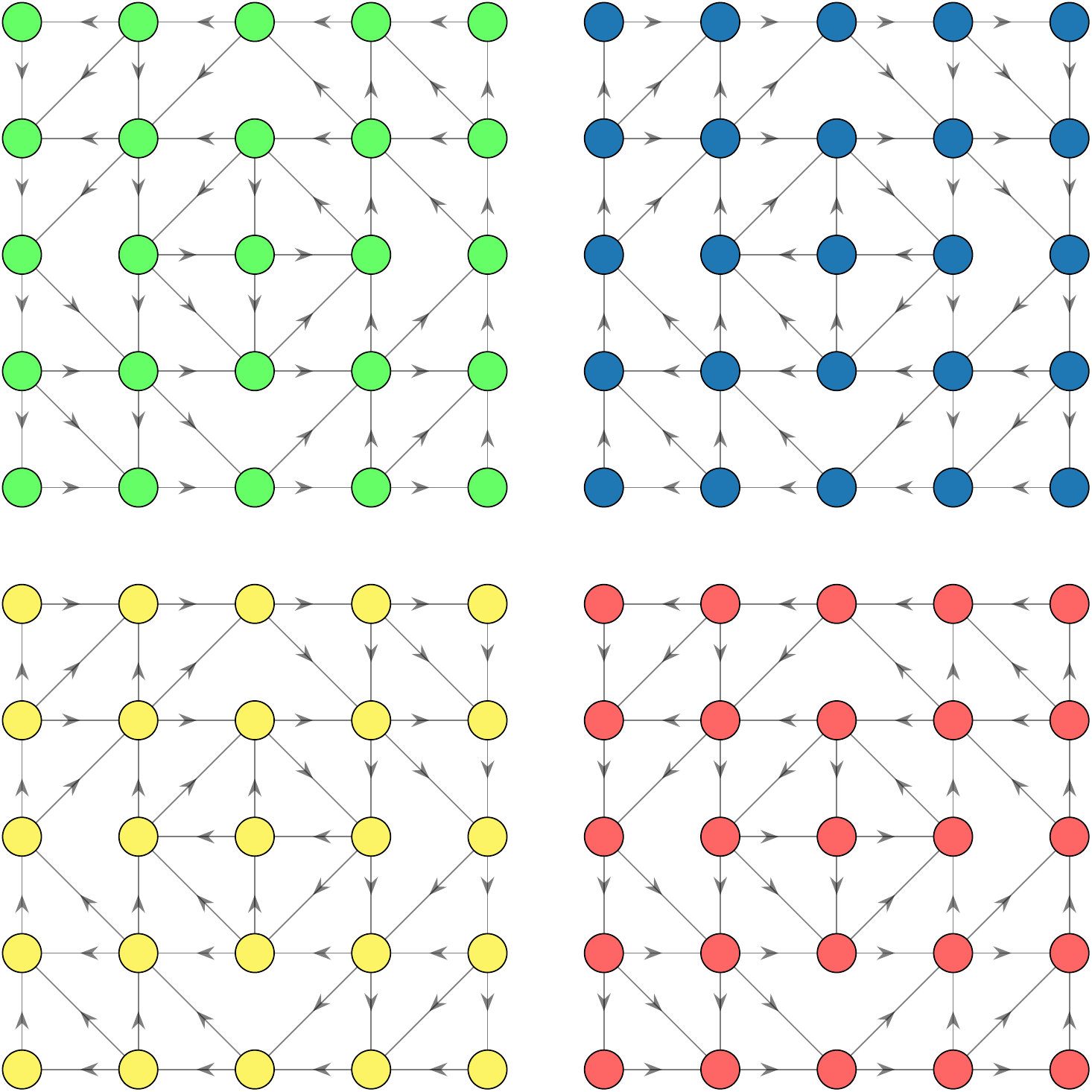}
    \end{minipage}
    \begin{minipage}[t]{0.49\textwidth}
        \centering
        \subfiguretitle{(j)}
        \vspace*{0.5ex}
        \includegraphics[width=0.7\linewidth]{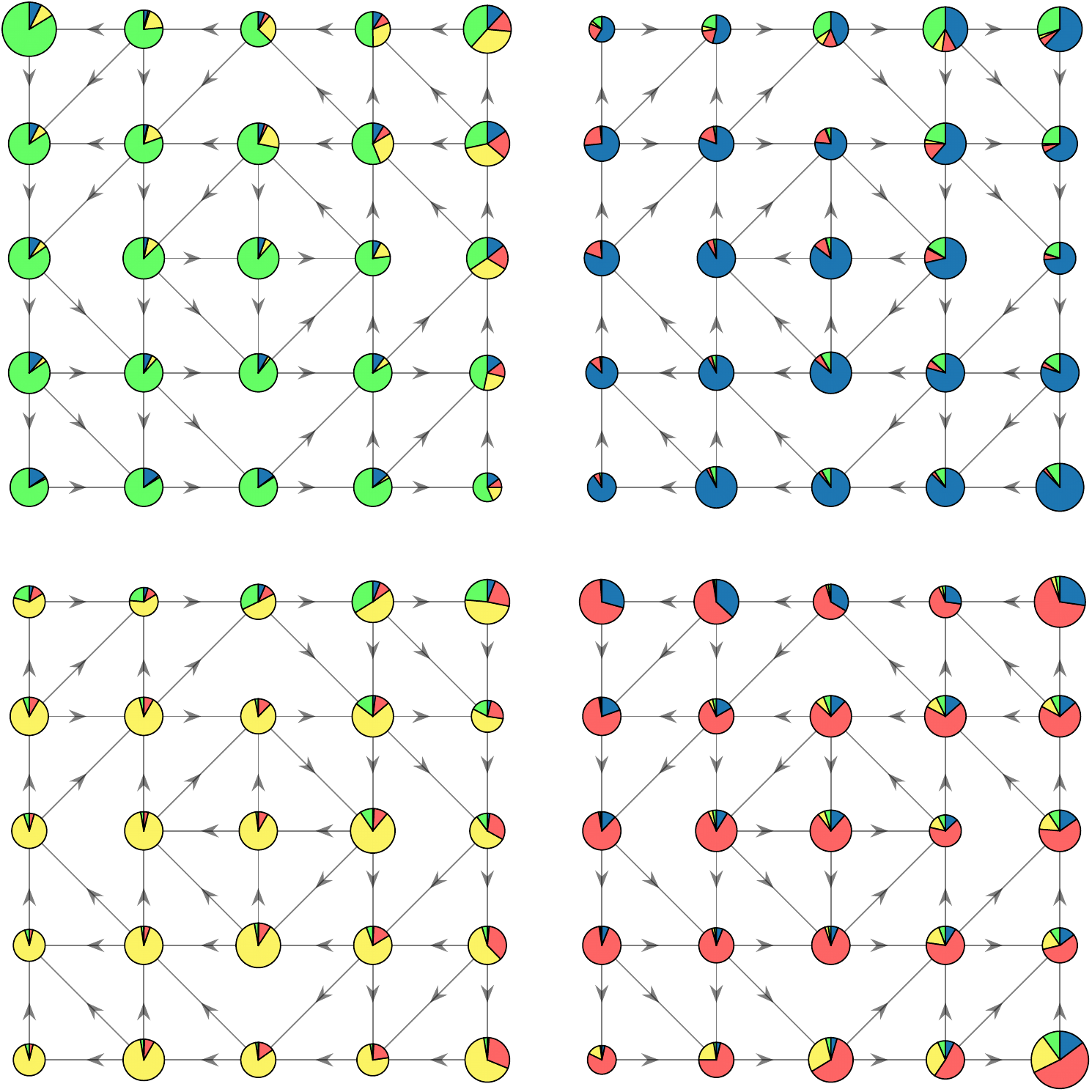}
    \end{minipage}
    \caption{(a) One long trajectory generated by the quadruple-well system defined in \eqref{eq:Quadruple-gyre flow}. (b)--(g)~Time-evolving quadruple-gyre graph at different times $ t $. (h) Eigenvalues of the matrix $ Q $. (i)~Clustering of the graph at time $ t = 0 $ into four sets. (j)~Clustering of the graph at time $ t = 20 $ obtained by mapping indicator functions for each individual cluster forward by the transition probability matrices. The resulting probabilities are plotted as pie charts, where the size of each vertex corresponds to the probability that a random walker will end up in this vertex at $ t = 20 $. It can be seen that information is leaking into the other clusters, but the influx and outflux is comparatively small, although some of the dominant eigenvalues are relatively close to zero. The identified sets are therefore coherent.}
    \label{fig:Quadruple-gyre graph}
\end{figure}

We create a time-evolving graph based on the quadruple-gyre system defined in \cite{DJM16}. Let $ f(t, z) = \delta \ts \sin(\omega \ts t) z^2 + (1 - 2 \ts \delta \sin(\omega \ts t)) z $ and
\begin{equation*}
    g(t, z_1, z_2) = \pi \ts \sin(\pi \ts f(t, z_1)) \cos(\pi \ts f(t, z_2)) \pd{f}{z_2}(t, z_2),
\end{equation*}
where $ \delta = 0.1 $ and $ \omega = 2 \ts \pi $. The quadruple-gyre flow on the 2-torus $ \mathbb{X} = [0, 2] \times [0, 2] $ is then given by
\begin{equation} \label{eq:Quadruple-gyre flow}
    \begin{split}
        \dot{x} = -g(t, x, y), \\\
        \dot{y} = \phantom{-} g(t, x, y).
    \end{split}
\end{equation}
The intersection point of the lines separating the four gyres that are either rotating clockwise or counterclockwise moves periodically along the diagonal, see \cite{DJM16}. We subdivide $ \mathbb{X} $ into $ 10 \times 10 $ equally sized boxes and select $ 16 $ test points per box, which are then mapped forward by the flow map $ \Theta^{0.05} $. Each box is represented by a vertex $ \mc[i]{v} $. If a test point is transported from box $ i $ to box $ j $, we add an edge $ (\mc[i]{v}, \mc[j]{v}) $. We simulate the system from $ \overline{t} = 0 $ to $ \overline{t} = 1 $, resulting in $ 20 $ graphs given by a sequence of adjacency matrices $ A^{(t)} $, $ t = 0, \dots, 19 $ (i.e., $ \overline{t} = 0.05 \ts t $), some of which are shown in Figure~\ref{fig:Quadruple-gyre graph}(b)--(g). At time $ t = 0 $, the clusters corresponding to the four gyres are disconnected, but as time increases transitions between clusters become possible. Next, we apply Approach B to compute the transition probability matrix $ P $ and  define $ Q = P D_\nu^{-1} P^\top $. Applying Algorithm~\ref{alg:directed} results in four coherent sets corresponding to the four vortices of the original quadruple-gyre system as shown in Figure~\ref{fig:Quadruple-gyre graph}(h)--(j). \exampleSymbol
\end{example}

In order to illustrate how the spectral clustering approach can be applied to more complex time-evolving networks, let us consider a high school contact and friendship network.

\begin{example} \label{ex:School network}

\begin{figure}[t]
    \centering
    \begin{minipage}[t]{0.19\textwidth}
        \centering
        \subfiguretitle{(a) $ t = 1 $}
        \vspace*{0.5ex}
        \includegraphics[width=0.95\linewidth]{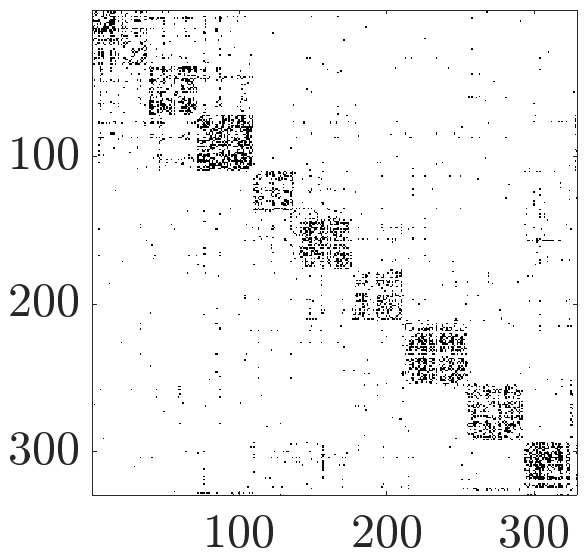}
    \end{minipage}
    \begin{minipage}[t]{0.19\textwidth}
        \centering
        \subfiguretitle{(b) $ t = 2 $}
        \vspace*{0.5ex}
        \includegraphics[width=0.95\linewidth]{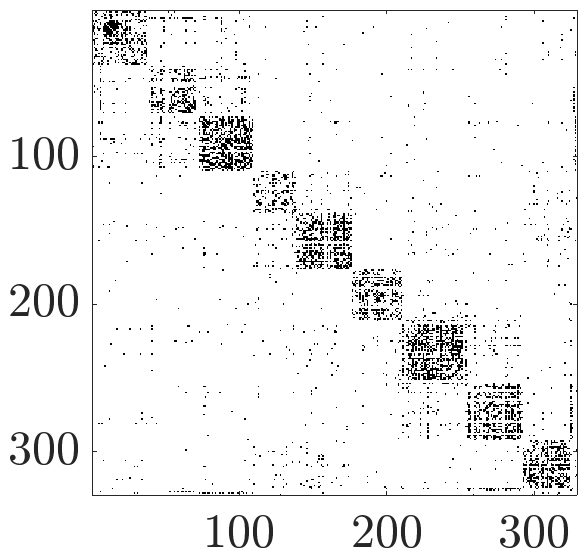}
    \end{minipage}
    \begin{minipage}[t]{0.19\textwidth}
        \centering
        \subfiguretitle{(c) $ t = 3 $}
        \vspace*{0.5ex}
        \includegraphics[width=0.95\linewidth]{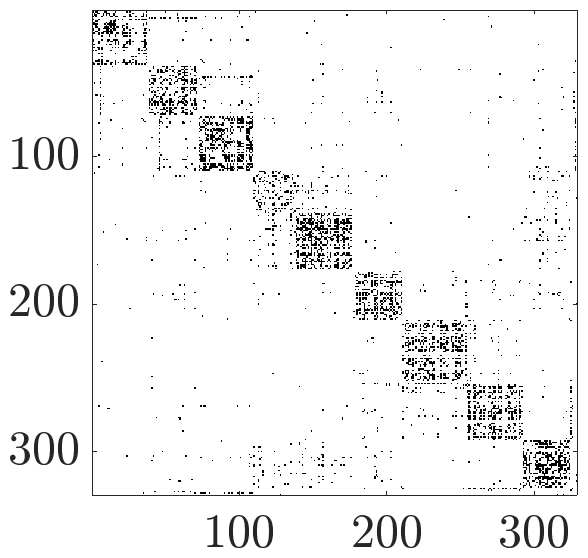}
    \end{minipage}
    \begin{minipage}[t]{0.19\textwidth}
        \centering
        \subfiguretitle{(d) $ t = 4 $}
        \vspace*{0.5ex}
        \includegraphics[width=0.95\linewidth]{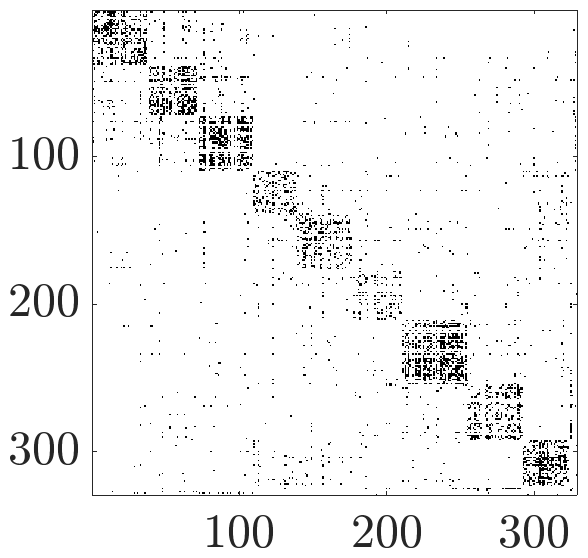}
    \end{minipage}
    \begin{minipage}[t]{0.19\textwidth}
        \centering
        \subfiguretitle{(e) $ t = 5 $}
        \vspace*{0.5ex}
        \includegraphics[width=0.95\linewidth]{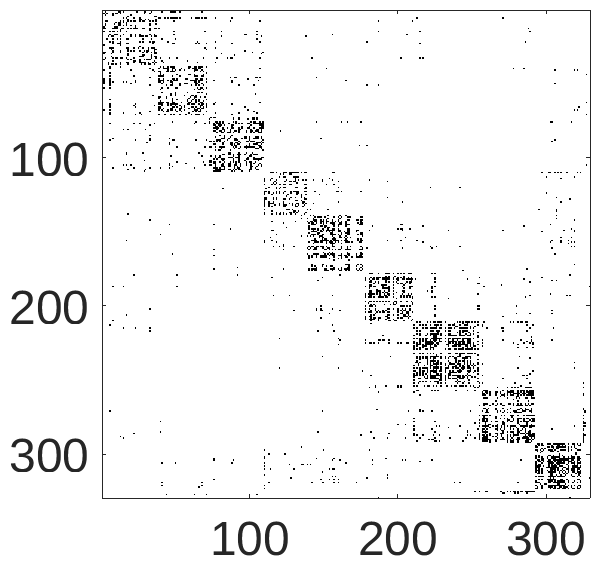}
    \end{minipage} \\[1ex]
    \begin{minipage}[t]{0.55\textwidth}
        \centering
        \subfiguretitle{(f)}
        \vspace*{0.5ex}
        \includegraphics[width=\linewidth]{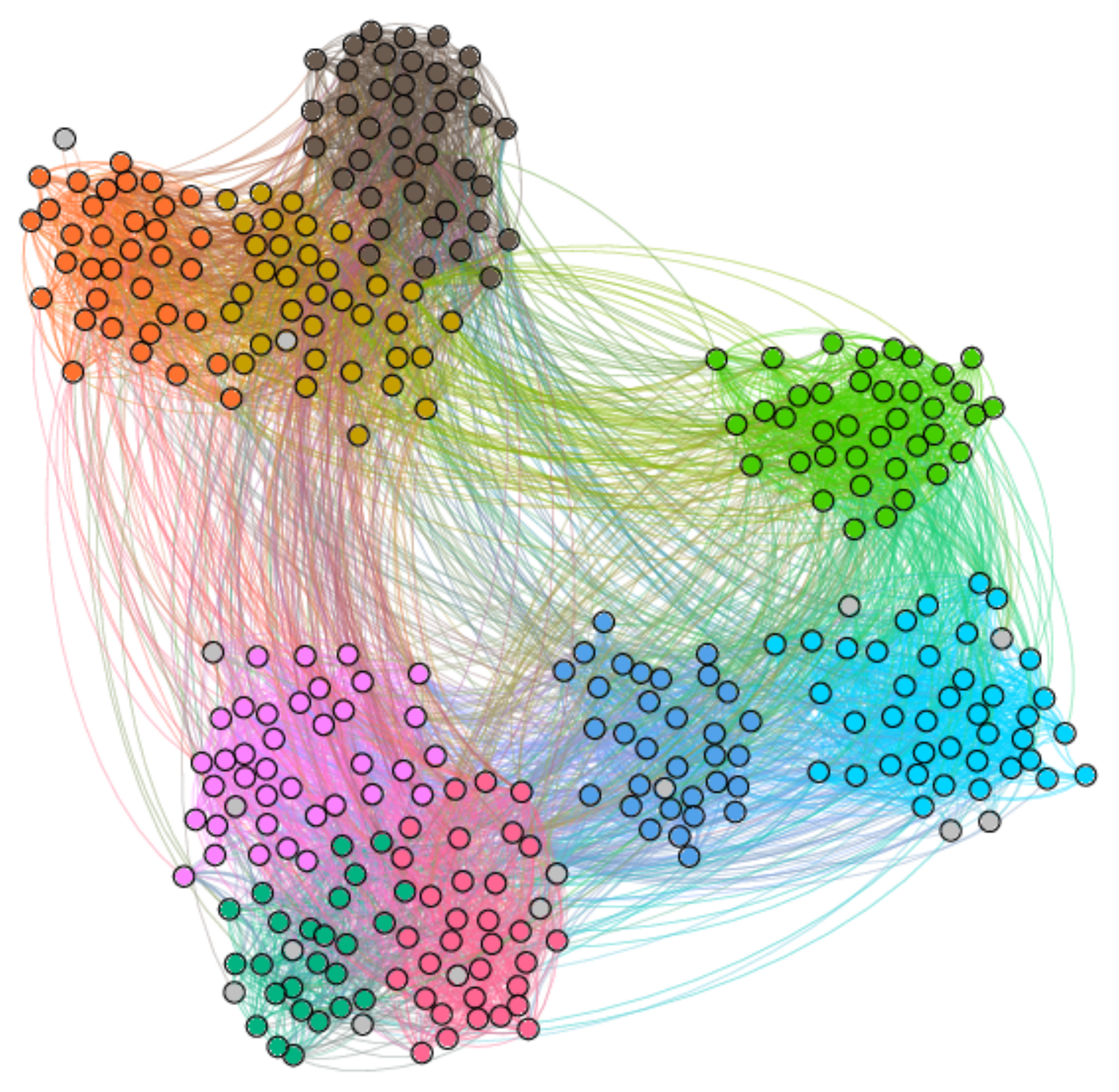}
    \end{minipage}
    \definecolor{gephi1}{RGB}{255, 113, 46}
    \definecolor{gephi2}{RGB}{197, 159, 0}
    \definecolor{gephi3}{RGB}{107, 90, 78}
    \definecolor{gephi4}{RGB}{0, 179, 128}
    \definecolor{gephi5}{RGB}{252, 131, 255}
    \definecolor{gephi6}{RGB}{82, 162, 235}
    \definecolor{gephi7}{RGB}{0, 213, 255}
    \definecolor{gephi8}{RGB}{71, 206, 0}
    \definecolor{gephi9}{RGB}{255, 103, 146}
    \caption{(a)--(e) Adjacency matrices of the time-evolving contact network on days one to five. (f) Spectral clustering of the graph. The clusters are:
    \protect\tikz[baseline] \protect\draw[black,fill=gephi1] (0,0.125) circle (.5ex);~2BIO1,
    \protect\tikz[baseline] \protect\draw[black,fill=gephi2] (0,0.125) circle (.5ex);~2BIO2,
    \protect\tikz[baseline] \protect\draw[black,fill=gephi3] (0,0.125) circle (.5ex);~2BIO3,
    \protect\tikz[baseline] \protect\draw[black,fill=gephi4] (0,0.125) circle (.5ex);~MP*1,
    \protect\tikz[baseline] \protect\draw[black,fill=gephi5] (0,0.125) circle (.5ex);~MP*2,
    \protect\tikz[baseline] \protect\draw[black,fill=gephi6] (0,0.125) circle (.5ex);~PSI*,
    \protect\tikz[baseline] \protect\draw[black,fill=gephi7] (0,0.125) circle (.5ex);~PC,
    \protect\tikz[baseline] \protect\draw[black,fill=gephi8] (0,0.125) circle (.5ex);~PC*, and
    \protect\tikz[baseline] \protect\draw[black,fill=gephi9] (0,0.125) circle (.5ex);~MP.
    The gray vertices have not been assigned to any of the clusters. Two vertices corresponding to students who did not have any contacts with their peers over the entire week are not displayed.}
    \label{fig:School network}
\end{figure}

We analyze data describing the social interactions of more than 300 high school students over a period of one week \cite{MFB15}. Face-to-face contacts were measured using wearable sensors that exchange IDs only when two students are facing each other. The data can be downloaded from \href{http://www.sociopatterns.org/datasets/high-school-contact-and-friendship-networks/}{www.sociopatterns.org}. We construct an undirected time-evolving graph by aggregating the contacts for each of the five consecutive days (smaller time intervals would be possible as well but result in very sparse graphs). The adjacency matrices, shown in Figure~\ref{fig:School network}(a)--(e), exhibit a cluster structure that corresponds to the different specializations: Mathematics and Physics (MP*1, MP*2, and MP), Physics and Chemistry (PC and PC*), Engineering (PSI*), and Biology (2BIO1, 2BIO2, and 2BIO3). We construct again a transition probability matrix $ P $ as described in Example~\ref{ex:Quadruple-gyre graph} (i.e., using Approach B), compute the forward-backward transition matrix $ Q $, and cluster the associated eigenvectors into $ 9 $ coherent sets. The resulting clustering is shown in Figure~\ref{fig:School network}(f).

The results are also summarized in Table~\ref{tab:SN results}. More than 93 \% (307/329) of the students are classified correctly. Approximately 5 \% (17/329) could not be assigned to a class. These are exactly the students who did not have any contacts with other students on the first day. Less than 2 \% (5/329) are misclassified. The incorrectly assigned students (IDs 274, 1543, 446, 9, and 784) either have only few contacts or more contacts with students from other classes. \exampleSymbol

\begin{table}
\renewcommand*{\arraystretch}{1.1}
\newcolumntype{Y}{>{\centering\arraybackslash}X}
\centering
\caption{The entry $(i, j)$ of the table shows how many students belonging to class $ i $ are assigned to class $ j $ by the spectral clustering algorithm. The column ``n/a'' contains the students who could not be assigned to coherent sets.}
\label{tab:SN results}
\scalebox{0.7}{
\begin{tabularx}{1.1\textwidth}{|c|*{10}{Y|} }
    \hline & & & & & & & & & & \\[-2ex]
          & 2BIO1 & 2BIO2 & 2BIO3 & MP*1 & MP*2 & PSI* & PC & PC* & MP & n/a \\[0.5ex] \hline
    2BIO1 &   35 &      &      &      &      &      &      &      &      &    1 \\
    2BIO2 &      &   33 &      &      &      &      &      &      &      &    2 \\
    2BIO3 &      &      &   40 &      &      &      &      &      &      &      \\
    MP*1  &      &      &      &   25 &    1 &      &      &      &      &    3 \\
    MP*2  &      &      &      &      &   35 &      &      &      &    1 &    2 \\
    PSI*  &      &      &      &      &    1 &   32 &      &      &      &    1 \\
    PC    &      &      &      &      &      &      &   39 &      &    1 &    4 \\
    PC*   &      &    1 &      &      &      &      &      &   38 &      &    1 \\
    MP    &      &      &      &      &      &      &      &      &   30 &    3 \\ \hline
\end{tabularx}}
\end{table}

\end{example}

This demonstrates that the spectral clustering algorithm correctly identifies clusters in time-evolving networks. Interestingly, Gephi's \cite{BHJ09} graph layout algorithm automatically places the Biology classes (2BIO1, 2BIO2, and 2BIO3) and Mathematics and Physics classes (MP*1, MP*2, and MP) in close proximity, which indicates that there is more interaction between students belonging to classes specializing in the same subject.

\section{Conclusion}
\label{sec:Conclusion}

We showed how spectral clustering of undirected graphs is related to computing eigenfunctions of the Koopman operator pertaining to the associated random walk process. For such reversible processes the operator is self-adjoint and the spectrum real-valued. If the process, however, is non-reversible or time-inhomogeneous, we in general obtain complex eigenvalues, which cannot be interpreted as relaxation time scales anymore. For such systems, the conventional concept of metastability no longer applies since metastable sets might now be time-dependent and move in state space. This leads to the definition of coherent sets, which have been extensively used to study fluid flows such as ocean or atmospheric dynamics. By defining transfer operators on graphs, it is possible to detect coherent sets in directed and time-evolving networks. These sets have the property that random walkers starting in such a cluster behave in a coherent way. Our approach can be regarded as a straightforward extension of the popular spectral clustering approach for undirected graphs. It follows the same steps, but replaces the standard random-walk Laplacian by a forward-backward counterpart. We illustrated that the generalized Laplacian leads to meaningful and interpretable clusters. Moreover, we showed how time-evolving benchmark graphs with intricate but intuitively accessible cluster structure can be constructed by discretizing time-inhomogeneous molecular dynamics or fluid dynamics problems. These graphs could, for instance, be used to compare various clustering algorithms for time-varying networks. Additionally, the \emph{probability leakage}---i.e., the amount of information escaping the clusters---could be used as a quality measure. This would allow for a systematic comparison of different clustering techniques, which will be considered in future work.

The next steps include analyzing the properties of the forward-backward Laplacian in detail and applying these methods to complex time-evolving graphs such as social networks in order to identify, for example, groups of users sharing similar behavioral patterns. Finding large-scale clusters in complex networks is often challenging and might necessitate fine-tuning and optimizing the proposed approach. Furthermore, it would also require efficient, reliable, and ideally easily parallelizable implementations of the algorithms. Another open question is how the chosen lag time $ \tau $ affects the detected coherent sets. The current (data-driven) spectral clustering method for time-evolving graphs takes only the start and end points of the trajectories into account. If the dynamics are smooth, this is often sufficient. More complex problems might benefit from multi-view CCA \cite{ShCh04}, which would take multiple time steps into account. Alternatively, the time-averaging approach using forward-backward diffusion matrices proposed in \cite{BaKo17} to construct space-time diffusion maps could be applied in our setting as well. These topics will be the subject of future research.

\section*{Data availability}

The data and code that support the findings of this study are openly available at \url{https://github.com/sklus/d3s/}.

\section*{Acknowledgments}

We would like to thank the reviewers for helpful comments and suggestions.

\bibliographystyle{unsrturl}
\bibliography{DGC}

\appendix

\section{Convergence proof}

\begin{proof}[Proof of Proposition~\ref{prop:Convergence}]
EDMD with indicator functions is equivalent to using Ulam's method~\cite{Ulam60}, see \cite{KKS16}. Let $ \delta_{ij} $ denote the Kronecker delta and $ \#\{ A \} $ the cardinality of the set $ A $. We have
\begin{align*}
    \big[C_{xx}\big]_{ij} &= \frac{1}{m} \sum_{k=1}^m \phi_i(x^{(k)}) \ts \phi_j(x^{(k)}) = \delta_{ij} \frac{\#\{ x^{(k)} = i \}}{m} \underset{\scriptscriptstyle m \rightarrow \infty}{\longrightarrow} \frac{\delta_{ij}}{n}, \\
    \big[C_{xy}\big]_{ij} &= \frac{1}{m} \sum_{k=1}^m \phi_i(x^{(k)}) \ts \phi_j(y^{(k)}) = \frac{\#\{ x^{(k)} = i \; \wedge \; y^{(k)} = j \}}{m} \underset{\scriptscriptstyle m \rightarrow \infty}{\longrightarrow} \frac{p_{ij}}{n}, \\
    \big[C_{yy}\big]_{ij} &= \frac{1}{m} \sum_{k=1}^m \phi_i(y^{(k)}) \ts \phi_j(y^{(k)}) = \delta_{ij} \frac{\#\{ y^{(k)} = j \}}{m} \underset{\scriptscriptstyle m \rightarrow \infty}{\longrightarrow} \frac{\delta_{ij}}{n} \sum_{\ell=1}^n p_{\ell j} = \frac{\delta_{ij}}{n} \ts \nu(\mc[j]{v}).
\end{align*}
It follows that $ \widehat{K}_\tau^{(m)} = C_{xx}^+ \ts C_{xy} \underset{\scriptscriptstyle m \rightarrow \infty}{\longrightarrow} P $ and that $ \widehat{P}_\tau^{(m)} = C_{xx}^+ \ts C_{yx} \underset{\scriptscriptstyle m \rightarrow \infty}{\longrightarrow} P^\top $. Similarly, $ \widehat{F}_\tau^{(m)} = C_{xx}^+ \ts C_{xy} \ts C_{yy}^+ \ts C_{yx} \underset{\scriptscriptstyle m \rightarrow \infty}{\longrightarrow} P D_\nu^{-1} P^\top = Q $, where we assume that $ D_\nu $ is invertible.
\end{proof}

\end{document}